\documentclass[11pt,a4paper,oneside]{amsart}
\usepackage[centertags]{amsmath}
\usepackage{bbm,amsfonts,amssymb, amssymb, amsthm}
\usepackage{indentfirst}
\usepackage{mathrsfs}
\usepackage{cancel}
\usepackage{newlfont}
\usepackage{comment}
\usepackage[all]{xy}

\usepackage{float}
\usepackage{color}
\newlength{\defbaselineskip}
\usepackage{booktabs}
\usepackage{tikz-cd}

\newtheorem{thm}{Theorem}[section]
\newtheorem{lem}[thm]{Lemma}
\newtheorem{prop}[thm]{Proposition}
\newtheorem{cor}[thm]{Corollary}

\newtheorem{rmk}[thm]{Remark}
\newtheorem{ex}[thm]{Example}
\newtheorem{defn}[thm]{Definition}

\numberwithin{equation}{section}
\newcommand{\Cour}[1]      {\left[\!\left[#1\right]\!\right]}

\newcommand{\id}           {\mathrm{id}}

\newcommand{\<}{\langle}
\renewcommand{\>}{\rangle}
\newcommand{\Ann} {\mathrm{Ann}}

\newcommand{\Lie}       {\mathcal{L}}
\newcommand{\frakx}     {\mathfrak{X}}
\newcommand{\G}         {\mathcal{G}}
\newcommand{\toto}      {\rightrightarrows}

\newcommand{\ttimes}   {_{\,s}\!\!\times_{t}}

\newcommand{\N}     {\mathcal{N}} 
\newcommand{\C}     {\mathbb{C}}   
\newcommand{\R}     {\mathbb{R}}
\newcommand{\T}     {\mathbb{T}}

\newcommand{\D}     {\mathbb{D}}

\newcommand{\A}     {\mathcal{A}}
\newcommand{\pr}    {\mathrm{pr}}
\newcommand{\ii}    {\mathbf{i}\,}
\newcommand{\Do}    {\overline{\partial}}
\renewcommand{\O}     {\mathcal{O}}
\newcommand{\tg}     {\scriptscriptstyle{tg}}
\newcommand{\cotg}     {\scriptscriptstyle{cotg}}
\newcommand{\tleft} {\langle\!\langle }
\newcommand{\tright}  {\rangle \! \rangle}

\usepackage[colorinlistoftodos]{todonotes}

\title{Dirac structures and Nijenhuis operators}
\author{Henrique Bursztyn} 
\address{IMPA, Estrada Dona Castorina 110, Rio de Janeiro, 22460-320, Brazil.}
\email{henrique@impa.br}
\author{Thiago Drummond}
\address{Departamento de Matem\'atica, Instituto de Matem\'atica,
Universidade Federal do Rio de Janeiro,
Caixa Postal 68530, Rio de Janeiro, RJ, 21941-909, Brasil.}
\email{drummond@im.ufrj.br}
\author{Clarice Netto}
\address{IMPA, Estrada Dona Castorina 110, Rio de Janeiro, 22460-320, Brazil.}
\email{claricenetto@id.uff.br}

\date{}

\begin{document}

\begin{abstract}
We introduce a notion of compatibility between (almost) Dirac structures and  $(1,1)$-tensor fields extending that of Poisson-Nijenhuis structures. 
We study several properties of the ``Dirac-Nijenhuis'' structures thus obtained,
including
their connection with holomorphic Dirac structures, the geometry of their leaves and quotients, as well as the presence of hierarchies.  We also consider their integration to Lie groupoids, which includes the integration of holomorphic Dirac structures as a special case.
\end{abstract}

\maketitle


\section{Introduction}

There are several situations of interest in geometry that involve the compatibility of a given geometric structure with a suitable $(1,1)$-tensor field. An illustrative example is the study of holomorphic objects on a complex manifold.
In this paper we introduce a notion of compatibility between (almost) Dirac structures and general $(1,1)$-tensor fields  that
extends, and is inspired by, the theory of {\em Poisson-Nijenhuis structures} \cite{KS-Magri90,MM}. In this context the compatibility  of a Poisson structure with a Nijenhuis operator has remarkable consequences, such as the presence of bi-hamiltonian structures \cite{GelfandDorfman,Mag}, Poisson hierarchies and integrals of motion that underpin many integrable systems, see e.g. \cite{BQT18,damfer08,DasOku,MagMars,MM}. In Poisson-Lie theory, Poisson-Nijenhuis structures (see \cite{rubtsov}) play a central role in certain quantization schemes \cite{bonechi15,BCQT14}.
Poisson-Nijenhuis structures also emerge naturally in the characterization of holomorphic Poisson structures in terms of their ``real parts'' \cite[$\S$~2.3]{LMX2}, and
this viewpoint is key to establish the holomorphic version of the integration of Poisson structures by symplectic groupoids \cite{SXC,SX} (see also \cite{crainic}). The notion of {\em Dirac-Nijenhuis structure} introduced in this paper provides a similar viewpoint to holomorphic Dirac structures and leads to more general integration results.


Dirac structures \cite{Cour} are common generalizations of Poisson structures and presymplectic forms, originally motivated by constrained mechanics (as the intrinsic geometry of submanifolds of Poisson manifolds), with a wide array of recent applications, see e.g. \cite{HB, Mein} and references therein. 
A Dirac structure on a manifold $M$ is a subbundle $L\subset \T M:= TM \oplus T^*M$ that is lagrangian (for the canonical symmetric pairing on $\T M$) and involutive with respect to the Courant-Dorfman bracket; any Poisson structure $\pi$ on $M$ may be regarded as a Dirac structure $L_\pi$ given by the graph of $\pi^\sharp: T^*M \to TM$, $\alpha\mapsto i_\alpha\pi$.  Dirac structures have also been considered in the holomorphic category, e.g. in connection with generalized K\" ahler geometry \cite{Gual14}  
and Poisson-Lie theory \cite{BIL}.

To better explain our results, recall that, in Poisson-Nijenhuis theory, one finds a notion of compatibility between a bivector field $\pi \in \mathfrak{X}^2(M)$ and a $(1,1)$-tensor field $r:TM \to TM$ that is expressed by two conditions:
\begin{itemize}
    \item[(1)] $\pi^\sharp \circ r^* = r \circ \pi^\sharp$,
    \item[(2)] For $X\in \mathfrak{X}(M)$ and $\alpha \in \Omega^1(M)$,
    $$
   R_\pi^r(X, \alpha) := \pi^\sharp(\Lie_X r^*(\alpha) - \Lie_{r(X)}\alpha) - (\Lie_{\pi^\sharp(\alpha)}r)(X) =0.
   $$
\end{itemize}
The expression in (2) is known as the {\em concomitant} of $\pi$ and $r$. The starting point of this paper is a new interpretation of these conditions that
clarifies how to extend them to Dirac structures.

By considering the map $(r,r^*): \T M \to \T M$, the purely algebraic condition (1) can be rephrased in terms of $L_\pi$ as $(r,r^*)(L_\pi)\subset L_\pi$.
Finding a parallel reformulation of the vanishing of the concomitant is less evident. For that, a central role is played by
two ``connection-like'' operators, $D^r: \mathfrak{X}(M) \to \Omega^1(M,TM)$ and $D^{r,*}: \Omega^1(M) \to \Omega^1(M,T^*M)$, given by
\begin{align*}
& D^r_X(Y) =  i_X(D^r(Y))  : = (\Lie_{Y}r)(X),\\
& D^{r,*}_X(\alpha) =  i_X(D^{r,*}(\alpha))  : =  \Lie_{X} r^*(\alpha) - \Lie_{r(X)}\alpha
\end{align*}
for $X,Y \in \mathfrak{X}(M)$, $\alpha\in \Omega^1(M)$. Setting $\D^r:= (D^r,D^{r,*})$, condition (2) is equivalent to $\D^r_X(\Gamma(L_\pi))\subseteq \Gamma(L_\pi)$ for all $X\in \mathfrak{X}(M)$ \cite[Sec.~2]{T}. This leads to the following definition: a  lagrangian subbundle $L\subset \T M$ is said to be {\em compatible} with a $(1,1)$-tensor field $r$ if
\begin{equation}\label{eq:DNcond}
(r,r^*)(L)\subset L, \qquad \D^r_X(\Gamma(L))\subseteq \Gamma(L),
\end{equation}
for all $X \in \mathfrak{X}(M)$.
When $L$ is a Dirac structure and $r$ is a Nijenhuis operator, the pair $(L,r)$ is called a {\em Dirac-Nijenhuis} structure. Besides recovering Poisson-Nijenhuis structures when $L$ is the graph of a Poisson structure, Dirac-Nijenhuis structures agree with the so-called {\em $\Omega N$-structures}\footnote{In this paper, we will refer to $\Omega N$-structures as  {\em presymplectic-Nijenhuis} structures.}  of \cite{MM}
when $L$ is the graph of a closed 2-form (see $\S$ \ref{sec:comp_Dirac}).
We stress that the notion of compatibility in \eqref{eq:DNcond} does not rely on the Courant-integrability of $L$ (or even the fact that it is lagrangian)  nor the Nijenhuis-integrability of  $r$.
This degree of generality makes it possible to accommodate, with no extra effort, weaker forms of integrability of $L$ and $r$, such as the ``quasi-Nijenhuis'' structures of \cite{SX} (see \cite{FMOP} for recent applications), see Remark \ref{rem:quasi}. Other approaches to Dirac-Nijenhuis structures can be found in the literature, see \cite{CG,Nunes04,LB,LB2}; we discuss how they compare with ours in Appendix~\ref{app:B}. As we will discuss in  \cite{BDC},
our Dirac-Nijenhuis structures fit into a more general theory of ``Courant-Nijenhuis algebroids''.

 We pass to the description of the main results and structure of the paper. We start by explaining in $\S$ \ref{sec:more} the general context in which the operators $D^r$ and $D^{r,*}$ naturally arise, namely that of {\em generalized derivations of degree 1} \cite[Sec.~2]{T}, or
{\em 1-derivations} for simplicity. Just as usual derivations of a vector bundle are in one-to-one correspondence with its linear vector fields  (see e.g. \cite[Thm.~3.4.5]{Mac-book}), 1-derivations on a vector bundle $E\to M$ are derivation-like objects that correspond to {\em linear} $(1,1)$-tensor fields on the total space $E$, see Theorem~\ref{thm:bijection}. For a $(1,1)$-tensor field $r$ on $M$, the triples $(D^r, r, r)$ and $(D^{r,*}, r^*, r)$ are examples of (dual) 1-derivations on $TM$ and $T^*M$, respectively; their corresponding  linear $(1,1)$-tensor fields are the so-called {\em tangent} and {\em cotangent lifts}  \cite{YI} of $r$, denoted by
$$
r^{tg}: T(TM)\to T(TM), \qquad r^{cotg}: T(T^*M)\to T(T^*M).
$$
Similarly, $(\D^r, (r,r^*), r)$  is a 1-derivation on $\T M$ associated with $(r^{tg}, r^{cotg}):T(\T M) \to T(\T M)$. With this correspondence at hand, the two conditions \eqref{eq:DNcond} in the definition of a Dirac-Nijenhuis structure $(L,r)$ are seen to be equivalent to the single condition
$$
(r^{tg}, r^{cotg})(TL)\subset TL.
$$
In the particular case of Poisson-Nijenhuis structures $(\pi,r)$, this last condition becomes
$$
T\pi^\sharp \circ r^{cotg} = r^{tg} \circ T\pi^\sharp,
$$
which provides a re-formulation of the compatibility conditions
(1) and (2) as one natural equation (see Proposition \ref{prop:tgcotPN}). It is also recalled in Example~\ref{ex:hol_str} that holomorphic structures on (real) vector bundles (regarded as partial flat connections as in \cite{Ra}) can be conveniently seen as particular 1-derivations. When $r$ is a complex structure on $M$, the $1$-derivations determined by $D^r$ and $D^{r,*}$  codify the holomorphic structures
on $TM$ and $T^*M$ arising from their usual identifications with $T^{1,0}M$ and $(T^{1,0}M)^*$, respectively (see Example~\ref{ex:holtancot}). By means of these  identifications,   we prove in Proposition \ref{prop:hol_dirac} that Dirac-Nijenhuis structures $(L, r)$, with $L\subset \T M$, are equivalent to holomorphic Dirac structures $\mathcal{L}\subset \T^{1,0}M = T^{1,0}M\oplus (T^{1,0}M)^*$.


After the definition and basic examples of Dirac-Nijenhuis structures are presented in \S \ref{sec:comp_Dirac}, in \S \ref{sec:properties} we discuss some of their fundamental properties, such as their behavior with respect to backward and forward Dirac maps (Propositions~\ref{prop:backdirac} and \ref{prop:forward_dirac}), the geometry of their presymplectic leaves  and Poisson quotients, as well as the existence of admissible functions in involution ($\S$ \ref{subsec:traces}) and hierarchies of Dirac structures ($\S$ \ref{subsec:hierarchy}).


The last two sections of the paper concern integration to Lie groupoids. A central result in Poisson geometry is the infinitesimal-global correspondence relating Poisson manifolds to objects known as {\em symplectic groupoids}, see e.g. \cite{catfel,CDW,CF} and references therein; the extension of this fact to Dirac geometry  is a correspondence between Dirac manifolds and  {\em presymplectic groupoids}  \cite{bcwz}. In this paper we enhance this picture by considering additional compatible $(1,1)$-tensor fields at infinitesimal and global levels.

As observed in $\S$ \ref{sec:multforms}, the notion of compatibility of a 2-form  with a given $(1,1)$-tensor field $r$ naturally extends to differential forms of arbitrary degrees
(one can in fact consider the compatibility of $r$ with arbitrary tensor fields through suitable extensions of the operators $D^r$ and $D^{r,*}$, though this is not discussed here), see Definition~\ref{dfn:nij_form}.  We consider Lie groupoids $\mathcal{G}$ equipped with a {\em compatible} pair $(\omega, K)$, where $\omega$ is a multiplicative differential form and $K$ is a multiplicative $(1,1)$-tensor field,
and provide their description at Lie-algebroid level in Theorem \ref{thm:compatibility_G}. As a consequence, when $K$ is a complex structure making $\mathcal{G}$ into a holomorphic Lie groupoid, we obtain in Theorem \ref{thm:hol_dif} a holomorphic version of \cite[Thm.~4.6]{HA} describing the infinitesimal counterparts of multiplicative differential forms. By means of these general results, we derive in $\S$ \ref{sec:int_dirac}  the infinitesimal-global correspondence between  Dirac-Nijenhuis structures and presymplectic-Nijenhuis groupoids in Theorem \ref{thm:main_integration} (see also Remark~\ref{rem:quasiint}), as well as the special case of this correspondence relating holomorphic Dirac structures and holomorphic presymplectic groupoids, see Theorems \ref{thm:diffpresgrp} and \ref{thm:integholDirac}. (Concrete constructions of holomorphic presymplectic groupoids integrating holomorphic Dirac structures arising in Poisson-Lie theory can be found in \cite{BIL}.)
Our differentiation-integration results are schematically illustrated  in the next diagram.

\bigskip


\begin{figure}[h!]
\label{fig1}
\centering

\tikzset{every picture/.style={line width=0.75pt}} 

\begin{tikzpicture}[x=0.75pt,y=0.75pt,yscale=-1,xscale=1]

\draw   (261.74,188.75) .. controls (261.74,185.06) and (264.73,182.08) .. (268.41,182.08) -- (349.77,182.08) .. controls (353.45,182.08) and (356.44,185.06) .. (356.44,188.75) -- (356.44,213.53) .. controls (356.44,217.21) and (353.45,220.2) .. (349.77,220.2) -- (268.41,220.2) .. controls (264.73,220.2) and (261.74,217.21) .. (261.74,213.53) -- cycle ;
\draw   (205.95,131.78) -- (227.64,145.87) -- (224.93,147.92) -- (246.71,177.18) -- (237.79,183.91) -- (216.01,154.65) -- (213.3,156.7) -- cycle ;
\draw   (337.24,49.82) -- (322.15,62.91) -- (322.15,56.36) -- (275.6,56.36) -- (275.6,43.28) -- (322.16,43.28) -- (322.16,36.74) -- cycle ;
\draw   (30.8,31.01) .. controls (30.8,19.27) and (40.32,9.75) .. (52.06,9.75) -- (233.76,9.75) .. controls (245.5,9.75) and (255.02,19.27) .. (255.02,31.01) -- (255.02,94.77) .. controls (255.02,106.51) and (245.5,116.03) .. (233.76,116.03) -- (52.06,116.03) .. controls (40.32,116.03) and (30.8,106.51) .. (30.8,94.77) -- cycle ;
\draw   (369.15,31.78) .. controls (369.15,19.72) and (378.93,9.94) .. (390.99,9.94) -- (598.18,9.94) .. controls (610.23,9.94) and (620.01,19.72) .. (620.01,31.78) -- (620.01,97.28) .. controls (620.01,109.34) and (610.23,119.11) .. (598.18,119.11) -- (390.99,119.11) .. controls (378.93,119.11) and (369.15,109.34) .. (369.15,97.28) -- cycle ;
\draw   (274.61,89.52) -- (289.7,76.43) -- (289.69,82.97) -- (336.24,82.97) -- (336.24,96.06) -- (289.69,96.06) -- (289.69,102.6) -- cycle ;
\draw   (410.64,132.07) -- (402.9,156.87) -- (400.22,154.78) -- (377.99,183.69) -- (369.18,176.82) -- (391.41,147.91) -- (388.74,145.82) -- cycle ;
\draw   (41.61,68.11) -- (238.6,68.11) -- (238.6,110.37) -- (41.61,110.37) -- cycle ;
\draw   (147.98,18.46) -- (238.6,18.46) -- (238.6,110.37) -- (147.98,110.37) -- cycle ;
\draw   (381.51,68.5) -- (607.46,68.5) -- (607.46,113.28) -- (381.51,113.28) -- cycle ;
\draw   (516.3,16.99) -- (607.46,16.99) -- (607.46,113.28) -- (516.3,113.28) -- cycle ;

\draw (264.67,194.29) node [anchor=north west][inner sep=0.75pt]  [font=\scriptsize] [align=left] {(1,1)-tensor field};
\draw (138.09,153.15) node [anchor=north west][inner sep=0.75pt]  [font=\tiny] [align=left] {Compatibility\\with Dirac \\structures};
\draw (279.64,25.3) node [anchor=north west][inner sep=0.75pt]  [font=\tiny] [align=left] {Integration};
\draw (42.84,70.99) node [anchor=north west][inner sep=0.75pt]  [font=\scriptsize] [align=left] {Dirac-Nijenhuis\\Structures\\};
\draw (382.66,73.28) node [anchor=north west][inner sep=0.75pt]  [font=\scriptsize] [align=left] {Presymplectic-Nijenhuis\\groupoids\\};
\draw (42.66,17.79) node [anchor=north west][inner sep=0.75pt]  [font=\scriptsize] [align=left] {IM forms \\compatible\\with 1-derivations\\};
\draw (378.72,20.81) node [anchor=north west][inner sep=0.75pt]  [font=\scriptsize] [align=left] {Multiplicative forms\\compatible with \\multiplicative (1,1)-tensors\\};
\draw (279.79,111.3) node [anchor=north west][inner sep=0.75pt]  [font=\tiny] [align=left] {Differentiation\\};
\draw (412.8,156.22) node [anchor=north west][inner sep=0.75pt]  [font=\tiny] [align=left] {Compatibility\\with differential \\forms};
\draw (153.49,33.34) node [anchor=north west][inner sep=0.75pt]  [font=\scriptsize] [align=left] {Holomorphic \\IM forms\\};
\draw (152.33,77.87) node [anchor=north west][inner sep=0.75pt]  [font=\scriptsize] [align=left] {Holomorphic \\Dirac structures\\};
\draw (523.21,69.74) node [anchor=north west][inner sep=0.75pt]  [font=\scriptsize] [align=left] {Holomorphic\\presymplectic\\groupoids\\};
\draw (523.33,23.27) node [anchor=north west][inner sep=0.75pt]  [font=\scriptsize] [align=left] {Holomorphic\\multiplicative \\forms};

\end{tikzpicture}

\caption{Compatibility vs. integration}
\end{figure}
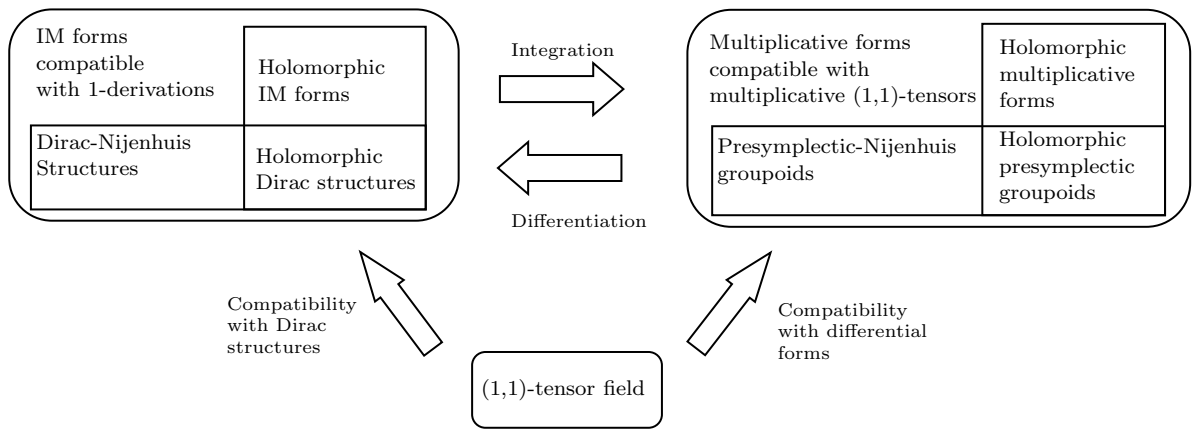


\bigskip

\noindent{\bf Acknowledgments}. 
We thank Pedro Frejlich, Jiang-Hua Lu, Juan Carlos Marrero and Igor Mencattini for stimulating discussions. We are grateful to CNPq, Capes and Faperj for financial support.

\tableofcontents

\section{Preliminaries: linear (1,1)-tensors and 1-derivations}\label{sec:more}

In this section we explain the correspondence between linear $(1,1)$-tensor fields on vector bundles and objects that we call {\em $1$-derivations}, a higher-degree analogue of the equivalence between linear vector fields and derivations on vector bundles. We then explain how the operators $D^r$ and $D^{r,*}$ naturally arise in this context.

\subsection{The correspondence}\label{subsec:1_tensors}

Let $q: E\to M$ be a (real) vector bundle, and consider its tangent bundle $TE\to E$. Recall that $TE$ carries a second vector-bundle structure $Tq: TE\to TM$ (its ``tangent prolongation'') making it a {\em double vector bundle} \cite{Mac-book}. We denote by $V_E=\mathrm{Ker}(Tq)\subset TE$ the vertical bundle of $E$.

A $(1,1)$-tensor field $K\in \Gamma(T^*E\otimes TE)=\Omega^1(E,TE)$ 
on the total space of $E$
is called {\em linear} if, regarded as a map $K: TE\to TE$, it is a vector-bundle morphism
$$
\xymatrix{
TE \ar[d] \ar[r]^{K} & TE \ar[d] \\
TM \ar[r]^{r} & TM,\\
}
$$
for some $r\in \Omega^1(M,TM)$. Note that linear $(1,1)$-tensors are higher-degree analogs of linear vector fields on $E$, 
which are elements  in $\Omega^0(E,TE)=\Gamma(TE)$ that, viewed as maps $E\to TE$, are vector-bundle morphisms from $E\to M$ to $TE\to TM$. (Equivalently, linear vector fields are characterized as those with flows given by vector-bundle automorphisms.)

Linear vector fields on $E$ are known to be in one-to-one correspondence with {\em derivations} of $E$, i.e., pairs  $(\Delta,X)$, where $\Delta:\Gamma(E)\to \Gamma(E)$ is $\mathbb{R}$-linear and $X\in \mathfrak{X}(M)$ is a vector field, satisfying the Leibniz rule
$$
\Delta(f u)= f\Delta(u) + (\Lie_Xf)u,
$$
for $u\in \Gamma(E)$ and $f\in C^\infty(M)$. One direct way to define the derivation $(\Delta,X)$ corresponding to a linear vector field $Z$ is via the relations
\begin{equation}\label{dfn:derivation}
\Delta(u)^\uparrow = [u^\uparrow, Z], \qquad  q^*\<\alpha, X\> = \< q^*\alpha, Z\>,
\end{equation}
for  $u\in\Gamma(E), \, \alpha \in \Omega^1(M)$; here 
$u^\uparrow \in \mathfrak{X}(E)$ is the {\em vertical lift } of $u\in \Gamma(E)$, 
$$u^\uparrow(e) = \frac{d}{d\tau}\Big|_{\tau=0}(e+\tau u(x)), \qquad \mbox{ for } e\in E|_x.
$$

As proven in \cite[Sec.~4.4]{HT}, this correspondence between linear vector fields and derivations is part of a general characterization of linear tensor fields on vector bundles, which establishes, in particular, a bijection between linear tensors of type $(1,p)$ (i.e., vector-valued forms) on $E$ and objects called {\em generalized derivations}, see \cite[Sec.~2]{T}. More specifically, linear $(1,1)$-tensor fields on $E$ bijectively correspond to {\em generalized derivations of degree 1} (or {\em 1-derivations} for simplicity), which are triples $(D,l,r)$, where $D: \Gamma(E)\to \Omega^1(M,E)$ is $\mathbb{R}$-linear, $l\in \mathrm{End}(E)$, $r\in \Omega^1(M,TM)$, satisfying the Leibniz-type rule
\begin{equation}\label{eq:genleibniz}
D(f u) = f D(u) + df\wedge l(u) -r^*(df) \wedge u,
\end{equation}
for $u\in \Gamma(E)$ and $f\in C^\infty(M)$. In particular, each $X\in \mathfrak{X}(M)$ gives rise to an operator $D_X: \Gamma(E)\to \Gamma(E)$, $D_X(u):= i_X(D(u))$, satisfying
\begin{equation}\label{eq:DXleibniz}
D_X(f u) = f D_X(u) + (\Lie_X f)l(u)-(\Lie_{r(X)}f)u. 
\end{equation}

We collect the main facts that we will need regarding linear $(1,1)$-tensors in the next result. Consider the map $\mathcal{V}: \Omega^1(M,E)\to \Omega^1(E,TE)$, $\alpha\otimes u \mapsto q^*\alpha \otimes u^\uparrow$, for $\alpha\in \Omega^1(M)$ and $u\in \Gamma(E)$.

\begin{thm} \label{thm:bijection}
\begin{itemize}
    \item[(i)] Any linear $(1,1)$-tensor $K$ on $E\to M$ defines a 1-derivation $(D,l,r)$ via 
    \begin{equation}\label{eq:1-1relations}
    \Lie_{u^\uparrow}K= \mathcal{V}(D(u)) ,\;\;\; K(u^\uparrow)= l(u)^\uparrow , \;\;\; 
    \langle K, q^*\beta \rangle = q^*\langle \beta, r \rangle  
    \end{equation}
 for $u\in \Gamma(E)$ and $\beta\in \Omega^1(M)$,  and this correspondence establishes a bijection between linear $(1,1)$-tensor fields and 1-derivations.
    
    \item[(ii)] For $i=1,2$, let $E_i\to M$ be a vector bundle equipped with a linear $(1,1)$-tensor field $K_i$, with corresponding 1-derivation $(D^i,l_i,r_i)$. Let $\Phi: E_1\to E_2$ be a vector-bundle map over the identity (and we keep the same notation for the induced map between spaces of sections). 
    Then $T\Phi\circ K_1 = K_2\circ T\Phi$ if and only if 
    \begin{equation}\label{eq:morprel}
    \Phi\circ D^1_X = D^2_X\circ \Phi,\;\;\; \Phi\circ l_1 = l_2\circ \Phi, \;\;\; r_1=r_2.
    \end{equation}
    for all $X\in \mathfrak{X}(M)$.
\end{itemize}
\end{thm}

Part (i) follows from  the general results in\cite{HT}. We present a proof of part (ii) in Appendix~\ref{subsec:proof}.
A consequence of the theorem (keeping the notation of part (i)) is that a subbundle $(F\to M)\subseteq (E\to M)$ satisfies $K(TF)\subseteq TF$ if and only if $l(F)\subseteq F$ and $D(\Gamma(F))\subseteq \Gamma(F)$.

\begin{ex}[Connections]{\em 
The one-to-one correspondence between linear $(1,1)$-tensors and $1$-derivations on $E\to M$ in Theorem~\ref{thm:bijection} (i) restricts to a bijection between linear $(1,1)$-tensors $\Theta: TE\to TE$ satisfying $\Theta^2=\Theta$ and $\mathrm{Im}(\Theta)=V_E$ and 1-derivations $(D,l,r)$ such that $l=\mathrm{id}_E$ and $r=0$. Note that any such $\Theta$ is completely determined by the choice of a linear subbundle (with respect to both linear structures on $E$) $H=\mathrm{Ker}(\Theta) \subset TE$ such that $TE=V_E\oplus H$, whereas $D: \Gamma(E)\to \Omega^1(M,E)$ defines a 1-derivation of the type $(D,\mathrm{id_E},0)$ if and only if it is a  connection operator on $E$. In other words, we recover the equivalence between two viewpoints to linear connections on $E$: via horizontal bundles or covariant derivatives, see e.g. \cite[$\S$11.11]{KMS}. (The operators $\Theta$ above are closely related to the so-called ``connectors''  of linear connections \cite[$\S$11.10]{KMS}.)}
\end{ex}

\begin{ex}[Holomorphic structures]\label{ex:hol_str}\em 
 A holomorphic vector bundle can be defined as a complex vector bundle over a complex manifold $M$ equipped with a flat $T^{0,1}M$-connection \cite{Ra}. Equivalently, one may regard a holomorphic vector bundle as a {\em real} vector bundle $E\to M$ together with a {\em holomorphic structure} consisting of
 (1) a complex structure  $r: TM\to TM$ on $M$, (2)
 a fiberwise complex structure on $E$, i.e., an endomorphism
 $l: E\to E$ with $l^2=-\id_E$,  and (3) a flat $T^{0,1}M$-connection $\nabla$ on the complex vector bundle $(E,l)$. The partial connection $\nabla: \Gamma(T^{0,1}M)\times \Gamma(E) \to \Gamma(E)$ is the same as a Dolbeault operator $\Do: \Gamma(E) \to \Gamma((T^{0,1}M)^*\otimes E)$, so we refer to it as a {\em Dolbeault connection}. A (local) section of $E$ is holomorphic if and only if it is flat. The holomorphic structure on $E$ may be alternatively represented by the triple $(D^{\mathrm{hol}}, l, r)$, where $D^{\mathrm{hol}}: \Gamma(E)\to \Omega^1(M,E)$ is given by
 \begin{equation}\label{eq:nablaD}
D^{\mathrm{hol}}_X(u) := l(\nabla_{X+\ii r(X)} u).
\end{equation}
One may verify that $(D^{\mathrm{hol}}, l, r)$ is in fact a 1-derivation on $E$, that we will call the  \textit{Dolbeault 1-derivation} associated with a holomorphic vector bundle. Such 1-derivations admit an intrinsic characterization as the 1-derivations $(D,l,r)$ on a vector bundle $E\to M$ satisfying the following set of equations (see \cite{HT}):
\begin{equation}\label{eq:J_square}
r^2= -\mathrm{id}_{TM}, \;\;\; l^2=-\mathrm{id}_{E}, \;\;\; l(D_X(u))+ D_{r(X)}(u)  =0,   
\end{equation}
as well as 
\begin{equation}\label{eq:IMNij}
\N_r(X,Y) = 0,\;\;\,
l(D_X(u) = D_X(l(u)),\;\;\,
l(D_{[X, Y]}(u)) = [D_X, D_Y](u) + D_{[X,Y]_r}(u),
\end{equation}
for all $X, Y \in \frakx(M)$ and $u \in \Gamma(E)$, where $\N_r$ is the Nijenhuis torsion of $r$ (see \eqref{eq:NijT} below) and
$[\cdot,\cdot]_r$ is given by $[X,Y]_r =[r(X),Y]+[X,r(Y)]-r([X,Y]$. Note that these equations just spell out conditions (1), (2) and (3) above, for $\nabla$ defined by $\nabla_{X + \ii r(X)} u = -l (D_X(u))$.
The linear (1,1)-tensor on $E$ associated to a Dolbeault 1-derivation  $(D^{\mathrm{hol}},l,r)$ via Theorem~\ref{thm:bijection} is the complex structure $J$ on the total space of $E\to M$, see  \cite[\S 6]{HT} and \cite[\S 5]{T} (here \eqref{eq:J_square} codifies the fact that $J^2=-\mathrm{id}_{TE}$ while conditions \eqref{eq:IMNij} give the vanishing of its Nijenhuis torsion).

\end{ex}

Analogously to ordinary derivations, any 1-derivation $(D,l,r)$ on $E$ gives rise to a {\em dual} 1-derivation on $E^*$ defined by the triple $(D^*,l^*,r)$, where $l^*: E^*\to E^*$ is dual to $l$, and $D^*$ is defined by the equation
\begin{equation}\label{eq:dualDgeral}
\<D^*_X(\xi), u \> = \Lie_X \<\xi, l(u)\> -\Lie_{r(X)}\<\xi,u\> - \<\xi,D_X(u)\>,     
\end{equation}
for $u\in \Gamma(E)$, $\xi\in \Gamma(E^*)$, and $X\in \mathfrak{X}(M)$. 

One can also describe this duality in terms of the corresponding linear $(1,1)$-tensor fields. For such  $K: TE \to TE$ with corresponding  1-derivation $(D,l,r)$, we denote by $K^\top: T(E^*) \to T(E^*)$ the linear $(1,1)$-tensor corresponding to $(D^*, l^*, r)$. 
Consider the non-degenerate pairing $\tleft \cdot,\, \cdot \tright: TE \times_{TM} T(E^*) \to TM\times \R$
obtained by differentiation of the natural pairing between $E$ and $E^*$. 
One can verify that, for every $(U, V) \in TE \times_{TM} T(E^*)$,
\begin{equation}\label{eq:K_dual_char}
\tleft K(U), K^\top(V) \tright = \tleft Tl(U), V \tright.
\end{equation}


\subsection{The operators $D^r$ and $D^{r,*}$}
\label{subsec:tancotan}
For a given $(1,1)$-tensor field $r$ on $M$,
consider the following nonlinear operators: $D^r: \mathfrak{X}(M) \to \Omega^1(M,TM)$,
\begin{align}\label{eq:Dr}
D^r_X(Y) =  i_X(D^r(Y)) & : = (\Lie_{Y}r)(X)\\ 
\nonumber
& = [Y, r(X)] - r([Y,X]), 
\end{align} 
for $X,Y \in \mathfrak{X}(M)$, and 
 $D^{r,*}: \Omega^1(M) \to \Omega^1(M,T^*M)$,
\begin{align}\label{eq:D*}
D^{r,*}_X(\alpha) =  i_X(D^{r,*}(\alpha)) & : =  \Lie_{X} r^*(\alpha) - \Lie_{r(X)}\alpha  \\ \nonumber
& = i_Xd(r^*\alpha) - i_{r(X)}d\alpha,
\end{align}
for $X \in \mathfrak{X}(M)$ and $\alpha\in \Omega^1(M)$. (The last equality in \eqref{eq:D*} follows from Cartan's formula.) 
We now recall from \cite[\S 3]{T} how these operators naturally arise as 1-derivations.

Any $(1,1)$-tensor field $r$ on $M$ gives rise to two  $(1,1)$-tensor fields, one on $TM$ and another on $T^*M$,
$$
r^{\tg} \in \Omega^1(TM, T(TM)),\qquad r^{\cotg}\in \Omega^1(T^*M,T(T^*M)),
$$
known as the {\em tangent} and {\em cotangent lifts} of $r$, see e.g. \cite{YI}.
The tangent lift $r^{\tg}: T(TM)\to T(TM)$ is obtained from the tangent map $Tr$ via $J_M\circ Tr \circ J_M$, where $J_M: T(TM)\to T(TM)$ is the canonical involution of the double tangent bundle \cite[\S 9.6]{Mac-book}. For the definition of the cotangent lift $r^{\cotg}$, consider the canonical symplectic structure $\omega_{\mathrm{can}}$ on the cotangent bundle $T^*M$, and let $\phi_r:= r^*: T^*M \to T^*M$. Then $r^{cotg}$ is determined by
$$
i_{r^{cotg}(U)}\omega_{\mathrm{can}} = i_U(\phi_r^*\,\omega_{\mathrm{can}}),
$$
see also \cite{Crampin}.

The next result is proven in \cite[\S 3]{T}.

\begin{prop}\label{prop:tancot1der}
For $r\in \Omega^1(M,TM)$, the $(1,1)$-tensor fields $r^{\tg}$ and $r^{\cotg}$ are linear and
satisfy $r^{cotg} = (r^{tg})^\top$. Their corresponding 1-derivations are $(D^r,r,r)$ and $(D^{r,*}, r^*, r)$.
\end{prop}

In particular, the operators $D^r$ and $D^{r,*}$ in \eqref{eq:Dr} and \eqref{eq:D*} satisfy the Leibniz-type identity in \eqref{eq:genleibniz} and are dual to one another as 1-derivations:
\begin{equation}\label{eq:dualD}
\<D^{r,*}_X(\alpha), Y \> = \Lie_X \<\alpha, r(Y)\> -\Lie_{r(X)}\<\alpha,Y\> - \<\alpha,D^r_X(Y)\>.     
\end{equation}

\begin{rmk}[Naturality properties]\label{rem:naturality}
{\em Consider $r_i\in \Omega(M_i, TM_i)$, $i=1,2$, and
a smooth map $\psi: M_1\to M_2$. Suppose that $r_1$ and $r_2$ are {$\psi$-related}, i.e., $r_2\circ d\psi = d\psi \circ r_1$. Let $X_i\in \mathfrak{X}(M_i)$, $i=1,2$, be vector fields which are also $\psi$-related;  we use the notation $X_1\sim_{\psi} X_2$. The following can be directly verified: 
\begin{align}
&Y_1\sim_\psi Y_2 \;\; \Rightarrow \;\; D^{r_1}_{X_1}(Y_1)\sim_\psi D^{r_2}_{X_2}(Y_2), \label{eq:natural1}
\\ 
& D^{r_1,*}_{X_1}(\psi^*\beta)=\psi^* D^{r_2,*}_{X_2}(\beta),\quad \forall\, \beta\in \Omega^1(M_2).\label{eq:natural2}
\end{align}
}
\hfill $\diamond$
\end{rmk}

We will be particularly interested in   \textit{Nijenhuis operators} $r\in \Omega^1(M,TM)$, i.e., endomorphisms $r:TM \to TM$ with vanishing Nijenhuis torsion $\N_r \in \Omega^2(M,TM)$:
\begin{equation}\label{eq:NijT}
\N_r(X, Y) := [r(X),r(Y)] - r([r(X),Y] + [X,r(Y)] - r[X,Y])=0.
\end{equation}

We note that the Nijenhuis torsion of $r\in \Omega^1(M,TM)$ admits the following expressions in terms of $D^r$ and $D^{r,*}$: for $X, Y \in \mathfrak{X}(M)$ and $\alpha\in \Omega^1(M)$,
\begin{align}
\label{Nij_D}\N_r(X,Y) & = r(D_X^r(Y)) - D^r_X(r(Y)) \\\label{eq:dualNij}
\<\alpha, \N_r(X,Y)\> & = \<r^*(D_X^{r,*}(\alpha))- D_X^{r,*}(r^*(\alpha)), Y\>.
\end{align}

The next example explains the meaning of the operators $D^r$ and $D^{r,*}$ when $r\in \Omega^1(M,TM)$ is a complex structure, i.e., $r$ is a Nijenhuis operator satisfying $r^2=-\id$.

\begin{ex}[Tangent and cotangent holomorphic structures]\label{ex:holtancot}
\em On a complex manifold $(M,r)$, $T^{1,0}M$ and $(T^{1,0}M)^*$ are holomorphic vector bundles. Explicitly, the Dolbeault connection on $T^{1,0}M$ is given by
$$
\nabla_{X^{\scriptscriptstyle{0,1}}}(Y^{\scriptscriptstyle{1,0}}) = \pr_{10}([X^{\scriptscriptstyle{0,1}}, Y^{\scriptscriptstyle{1,0}}]), 
$$
where $\pr_{10}$ is the projection of $TM\otimes \mathbb{C}= T^{1,0}M\oplus T^{0,1}M$ onto $T^{1,0}M$, 
while $(T^{1,0}M)^*$ carries the dual Dolbeault connection, defined by the condition
\begin{equation}\label{eq:dualDol}
\<\nabla^*_{X^{\scriptscriptstyle{0,1}}}(\alpha^{\scriptscriptstyle{1,0}}), Y^{\scriptscriptstyle{1,0}}\>  = \Lie_{X^{\scriptscriptstyle{0,1}}}\<\alpha^{\scriptscriptstyle{1,0}}, Y^{\scriptscriptstyle{1,0}}\> - \<\alpha^{\scriptscriptstyle{1,0}}, \nabla_{X^{\scriptscriptstyle{0,1}}}(Y^{\scriptscriptstyle{1,0}})\>,
\end{equation}
where $X^{\scriptscriptstyle{0,1}} \in \mathfrak{X}^{0,1}(M)$,  $Y^{\scriptscriptstyle{1,0}} \in \mathfrak{X}^{1,0}(M)$, and $\alpha^{\scriptscriptstyle{1,0}} \in \Omega^{1,0}(M)$.
Consider the usual identifications (of real vector bundles)
\begin{equation}\label{eq:identifPhi}
\Phi_T: TM\to T^{1,0}M, \; X\mapsto \frac{1}{2}(X-\ii r(X)),\qquad \Phi_{T^*}: T^*M \to (T^{1,0}M)^*,\;  \alpha \mapsto \alpha - \ii r^*(\alpha). 
\end{equation}
One can then verify (see \cite[Prop.~5.3]{T}) that the Dolbeault 1-derivations codifiying the holomorphic structures on $TM$ and $T^*M$ resulting from these identifications are $(D^r,r,r)$ and $(D^{r,*},r^*,r)$; in particular, by Proposition~\ref{prop:tancot1der}, the complex structures on $TM$ and $T^*M$ are $r^{\tg}$ and $r^{\cotg}$, respectively.
\end{ex}

\section{Compatibility of Dirac structures and $(1,1)$-tensor fields}
\label{sec:comp_Dirac}
We introduce in this section a notion of compatibility for (almost) Dirac structures and $(1,1)$-tensor fields,
extending the well-known case of Poisson-Nijenhuis manifolds that we recall first.

\subsection{The Poisson-Nijenhuis compatibility revisited}\label{subsec:PNrevisit}

On a manifold $M$, let $\pi \in \mathfrak{X}^2(M)$ be a bivector field and $r\in \Omega^1(M,TM)$.
Consider the bundle map $\pi^\sharp: T^*M\to TM$, $\alpha\mapsto i_\alpha \pi$.

\begin{defn}\label{def:rpi}
We say that the pair $(\pi,r)$ is {\em compatible} if the following conditions hold:
\begin{itemize}
\item[(1)] $\pi^\sharp \circ r^*  = r \circ \pi^\sharp$;
\item[(2)] For $X \in \frakx(M)$, $\alpha \in \Omega^1(M)$, we have
\begin{equation}\label{dfn:mm_concomitant}
R_\pi^r(X, \alpha) := \pi^\sharp(\Lie_X r^*(\alpha) - \Lie_{r(X)}\alpha) - (\Lie_{\pi^\sharp(\alpha)}r)(X) =0.
\end{equation}
\end{itemize}
\end{defn}

Note that by (1) we have another bivector field $\pi_r \in \frakx^2(M)$ defined by the condition  
$$
\pi_r^\sharp = r \circ \pi^\sharp.
$$
We also notice that $R_\pi^r: \mathfrak{X}(M)\times \Omega^1(M)\to \mathfrak{X}(M)$ is $C^\infty(M)$-bilinear if and only if  condition (1) holds; in this case $R_\pi^r$ defines a tensor field called  the \textit{Magri-Morosi concomitant} \cite{MM}.

\begin{rmk}\label{rem:alternatives}
{\em In the literature, condition (2) above is often expressed through the vanishing of  an alternative concomitant, $C^r_\pi: \Omega^1(M)\times \Omega^1(M) \to \Omega^1(M)$, given by
\begin{equation}\label{dfn:schw_concomitant}
C_\pi^r(\alpha, \beta) := [\alpha,\beta]_{\pi_r} - ([r^*(\alpha), \beta]_\pi + [\alpha, r^*(\beta)]_\pi - r^*([\alpha, \beta]_\pi)), 
\end{equation}
where $[\cdot,\cdot]_\Lambda$ is the bracket on $\Omega^1(M)$ defined by a bivector field $\Lambda$ via
$[\alpha, \beta]_\Lambda := \Lie_{\Lambda^\sharp(\alpha)} \beta - i_{\Lambda^\sharp(\beta)}d\alpha$.
The concomitants $R^r_\pi$ and $C^r_\pi$ are related by
$$
\<\beta, R_\pi^r(X, \alpha)\> = \<C_\pi^r(\alpha, \beta), X\>,
$$
so $R_\pi^r = 0$ if and only if $C_\pi^r = 0$. \hfill $\diamond$}
\end{rmk}

A \textit{Poisson-Nijenhuis structure} (or simply a {\em PN structure}) on a manifold $M$, as introduced in \cite{MM} (see also \cite{KS-Magri90}),  is a compatible pair $(\pi, r)$, where  $\pi$ is a Poisson bivector field and $r$ is a Nijenhuis operator. 
In this setting,  a geometric interpretation of the compatibility conditions in terms of Lie bialgebroids can be found in \cite{Kos96}. 

Our extension of the compatibility condition of PN structures to Dirac structures relies on a suitable reformulation of condition (2) above, see \cite[\S 2.3]{T}. Let $\pi$ be a bivector field on $M$ and $r\in \Omega^1(M,TM)$. Consider the operators $D^r$ and $D^{r,*}$ defined in \eqref{eq:Dr} and \eqref{eq:D*}.
Then the Magri-Morosi concomitant can be written as
\begin{equation}\label{eqn:mm_alt}
R_\pi^r(X,\alpha) = \pi^\sharp(D^{r,*}_X(\alpha)) - D^r_X(\pi^\sharp(\alpha)). 
\end{equation}

As a consequence of Prop.~\ref{prop:tancot1der} and Theorem~\ref{thm:bijection}, part (ii), 
we have the following alternative description  of the compatibility of the pair $(\pi,r)$ in terms of tangent and cotangent lifts:

\begin{prop}\label{prop:tgcotPN}
The pair $(\pi,r)$ is compatible if and only if $T\pi^\sharp\circ r^{\cotg} = r^{\tg}\circ T\pi^\sharp$.
\end{prop}

 
\subsection{The presymplectic-Nijenhuis compatibility revisited}\label{subsec:omegaN}

Given $\omega\in \Omega^2(M)$ and $r\in \Omega^1(M,TM)$, we now introduce a notion of compatibility that is parallel to the one in Section~\ref{subsec:PNrevisit}.

Consider $\omega^\flat: TM\to T^*M$, $X \mapsto i_X \omega$, and define (in analogy with \eqref{eqn:mm_alt})
\begin{align}\label{eq:Stilde}
\widetilde{S}^r_\omega(X,Y) &:=  D^{r,*}_X(\omega^\flat(Y))-\omega^\flat(D^r_X(Y))
 \\
& =
i_Xd(r^*\circ \omega^{\flat}(Y))-   i_{r(X)}d\omega^{\flat}(Y) - \omega^{\flat}([Y,r(X)]-r([Y,X])), \nonumber
\end{align}
for $X$, $Y \in \mathfrak{X}(M)$. 

Suppose that $\omega^\flat\circ r = r^* \circ \omega^\flat$,  so that we have $\omega_r \in \Omega^2(M)$ given by 
$$
i_X\omega_r:=i_{r(X)}\omega.
$$
Let $(d\omega)_r\in \Gamma(T^*M\otimes \wedge^2 T^*M)$ be given by
$i_X(d\omega)_r=i_{r(X)} d\omega$.

\begin{lem}
\label{lem:Stilde}
If $\omega^\flat\circ r = r^* \circ \omega^\flat$,  then $\widetilde{S}^r_\omega = d(\omega_r) - (d\omega)_r$.
\end{lem}

\begin{proof}
The expression for $\widetilde{S}^r_\omega$ in \eqref{eq:Stilde} can be re-written as follows:
$$
\widetilde{S}^r_\omega(X,Y) = i_Xdi_Y \omega_r - i_{r(X)}di_Y\omega -i_{[Y,r(X)]}\omega + i_{[Y,X]}\omega_r.
$$
From the identities 
\begin{align*}
& i_Xdi_Y \omega_r = i_X\Lie_Y\omega_r-i_Xi_Y d(\omega_r), &i_{r(X)}di_Y\omega = i_{r(X)}\Lie_Y\omega - i_{r(X)}i_Yd\omega\\ 
& i_{[Y,r(X)]}\omega = \Lie_{Y}i_{r(X)}\omega-i_{r(X)}\Lie_{Y}\omega , & i_{[Y,X]}\omega_r = \Lie_Yi_X\omega_r -i_X\Lie_Y\omega_r,
\end{align*}
we see that
$$
\widetilde{S}^r_\omega(X,Y)=-i_Xi_Yd(\omega_r)+i_{r(X)}i_Yd\omega = i_Yi_Xd(\omega_r) - i_Yi_X(d\omega)_r.
$$
\end{proof}

\begin{defn}\label{def:romega}
For $\omega\in \Omega^2(M)$ and $r\in \Omega^1(M,TM)$, we say that the pair $(\omega,r)$ is {\em compatible} if
\begin{itemize}
    \item[(1)]  $\omega^\flat \circ r=  r^* \circ \omega^\flat$,
    \item[(2)] $\widetilde{S}^r_\omega=0$ (equivalently, $d(\omega_r)=(d\omega)_r$). 
\end{itemize}
\end{defn}

In analogy with Proposition~\ref{prop:tgcotPN}, this notion of compatibility has the following interpretation:

\begin{prop}\label{prop:tancotOmN}
The pair $(\omega,r)$ is compatible if and only if $r^{\cotg} \circ T\omega^\flat = T\omega^\flat \circ r^{\tg}$.
\end{prop}

Alongside Poisson-Nijenhuis structures,
Magri and Morosi considered in \cite{MM} 
  {\em presymplectic-Nijenhuis structures} (called {\em $\Omega N$-structures} therein), which were defined as compatible pairs $(\omega,r)$, where $\omega\in \Omega^2(M)$ is closed and $r \in \Omega^1(M,TM)$ is Nijenhuis. In this case, condition (2) above reduces to $d(\omega_r)=0$.

\begin{rmk}\label{rem:S}\em
In the original formulation of \cite{MM}, condition $(2)$ above was written as the vanishing of a concomitant $S^r_\omega\in \Omega^2(M,T^*M)$ given by 
\begin{equation*}\label{eq:S}
S_\omega^r(X, Y):= i_X\mathcal{L}_{r(Y)} \omega - i_Y\mathcal{L}_{r(X)}\omega - i_{r([X,Y])}\omega + d(\omega(r(Y),X)).
\end{equation*}
Using condition (1), a direct application of Cartan calculus
shows that  
\begin{equation*}\label{eq:S_omega_rel}
S^r_\omega(X,Y)= i_Yi_Xd(\omega_r) - i_Yi_{r(X)}d\omega - i_{r(Y)}i_X d\omega.
\end{equation*}
When $\omega$ is closed, we see that 
$S^r_\omega = d(\omega_r) = \widetilde{S}^r_\omega$. In general, the concomitant $S^r_\omega$ does not agree with our $\widetilde{S}^r_\omega$ in \eqref{eq:S}; they are related by $\widetilde{S}^r_\omega(X,Y)= S^r_\omega(X,Y) + d\omega(X, r(Y))$. 
\hfill $\diamond$
\end{rmk}

\subsection{Dirac-Nijenhuis structures} 
Let $M$ be a smooth manifold, and consider  
the direct sum $\mathbb{T}M =TM \oplus T^*M$
with natural projections $\pr_T: \T M \to TM$ and $\pr_{T^*}: \T M \to T^*M$. We equip $\T M$ with
the nondegenerate fiberwise symmetric pairing
$$
\<(X,\alpha), (Y, \beta)\> = \beta(X) + \alpha(Y),
$$
and the Courant-Dorfman bracket $\Cour{\cdot,\cdot}$ on $\Gamma(\mathbb{T}M)=\mathfrak{X}(M)\oplus\Omega^1(M)$,
$$
\Cour{(X,\alpha), (Y,\beta)} = ([X,Y], \Lie_X \beta - i_Y d\alpha).
$$

A subbundle $L\subset \mathbb{T}M$ is called {\em lagrangian} if $L=L^\perp$ with respect to $\<\cdot,\cdot\>$. A \textit{Dirac structure} \cite{Cour} on $M$ is a lagrangian subbundle $L \subset \T M$ which is involutive with respect to $\Cour{\cdot, \cdot}$. Important examples
of lagrangian subbundles include bivector fields $\pi$ and 2-forms $\omega$, viewed as subbundles of $\T M$ by means of the graphs of the 
corresponding maps $\pi^\sharp: T^*M\to TM$ and $\omega^\flat: TM\to T^*M$; the additional involutivity condition to make these subbundles into Dirac structures amount to $\pi$ being Poisson and $\omega$ being closed.

Given $r\in \Omega^1(M,TM)$,
consider the operator 
\begin{equation}\label{eq:sumDr}
\mathbb{D}^r: \Gamma(\mathbb{T}M)\to \Omega^1(M,\mathbb{T}M), \qquad \mathbb{D}^r=(D^r,D^{r,*}).
\end{equation}
For each vector field $X\in \mathfrak{X}(M)$ we then have an $\R$-linear operator $\mathbb{D}^r_X: \Gamma(\mathbb{T}M) \rightarrow \Gamma(\mathbb{T}M )$,
\begin{equation}\label{dfn:big_D}
\mathbb{D}^r_X(Y,\alpha)=( D^{r}_X(Y), D^{r,*}_X(\alpha)),
\end{equation}
for $(Y,\alpha) \in \Gamma(\mathbb{T}M)=\mathfrak{X}(M)\oplus \Omega^1(M)$. 
From the Leibniz-type equations for $D^r$ and $D^{r,*}$, we see that $\mathbb{D}^r$ satisfies 
\begin{equation}\label{eq:leibniz_D}
\D^r_X(f \,\sigma) = f \,\D(\sigma) + (\Lie_Xf)(r,r^*)(\sigma) - (\Lie_{r(X)}f)\,\sigma,
\end{equation}
for $\sigma \in \Gamma(\T M)$, and $f \in C^{\infty}(M)$, so the triple $(\mathbb{D}^r, (r,r^*), r)$ is a 1-derivation on $\T M$.

\begin{defn}\em \label{defDN} Let $L \subset \mathbb{T}M$ be a lagrangian subbundle and $r \in \Omega^1(M,TM)$. We say that the pair $(L,r)$ is \textit{compatible} if 
\begin{itemize}
    \item [$(1)$] $(r,r^*)(L) \subseteq L$
    \item [$(2)$] $\mathbb{D}^r_X(\Gamma(L))\subseteq \Gamma(L), \: \forall X \in \mathfrak{X}(M).$
\end{itemize}
\end{defn}

(Note that this notion of compatibility makes sense for any subbundle of $\T M$, not necessarily lagrangian.)

The previous definition has an interpretation that naturally extends Propositions~\ref{prop:tgcotPN} and \ref{prop:tancotOmN}: a pair $(L,r)$ is compatible if and only if 
$$
TL\subset T(TM\oplus T^*M) = T(TM)\times_{TM} T(T^*M)
$$ 
satisfies $(r^{\tg},r^{\cotg})(TL)\subset TL$. 

\begin{rmk}\label{rem:concom} {\em 
Condition (2) in Definition~\ref{defDN} can  be also written in terms of the vanishing of a concomitant: for a lagrangian subbundle $L\subset \T M$ satisfying $(r,r^*)(L) \subseteq L$, the expression 
\begin{equation}
C^r_L(\sigma_1,\sigma_2) := \<\D^r_{(\cdot)}(\sigma_1), \sigma_2\> \in \Omega^1(M), \qquad \sigma_1,\,\sigma_2 \in \Gamma(L),
\end{equation}
defines an element in $\Gamma(\wedge^2 L^* \otimes T^*M)$. (The skew-symmetry of $C^r_L$ is a consequence of the duality \eqref{eq:dualD}, and the $C^\infty(M)$-linearity follows from \eqref{eq:leibniz_D}.) The fact that $L$ is lagrangian implies that (2) holds if and only if $C^r_L = 0$.}
\hfill $\diamond$
\end{rmk}

\begin{defn}
A \textit{Dirac-Nijenhuis structure} on $M$ is a compatible pair $(L,r)$ where $L$ is a Dirac structure and $r$ is a Nijenhuis operator. 
\end{defn}

More general Dirac-Nijenhuis structures in abstract ``Courant-Nijenhuis algebroids'' will be the subject of \cite{BDC}.

\begin{ex}\label{exam:pn_dn}\em
Let $\pi \in \frakx^2(M)$ and consider the lagrangian subbundle $L_\pi := \mathrm{graph}(\pi^\sharp)\subset \T M$. Given $r\in \Omega^1(M,TM)$, then $(r,r^*)(L_\pi) \subseteq L_\pi$ if and only if $\pi^\sharp \circ r^* = r\circ \pi^\sharp$. It follows from \eqref{eqn:mm_alt} that $\D^r_X(\Gamma(L_\pi))\subset \Gamma(L_\pi)$ if and only if $R_\pi^r=0$. Hence the pair $(L_\pi,r)$ is compatible if and only if so is the pair $(\pi, r)$. (Under the natural identification $L_\pi \cong T^*M$,  the concomitants $C_{L_\pi}^r$ and $C_\pi^r$, see Remarks~\ref{rem:alternatives} and \ref{rem:concom}, are in fact identified.) In particular, $(\pi,r)$ is a Poisson-Nijenhuis structure if and only if $(L_\pi,r)$ is Dirac-Nijenhuis.
\end{ex}

Other examples of Dirac-Nijenhuis structures can be found on submanifolds of Poisson-Nijenhuis structures, see Corollary \ref{cor:DNsubm}.

\begin{ex} \label{ex:OmegaN}
{\em For $\omega\in \Omega^2(M)$, let $L_\omega= \mathrm{graph}(\omega^\flat) \subset \T M$ be the corresponding lagrangian subbundle. Let $r\in \Omega^1(M,TM)$. Then  condition $(r,r^*)(L_\omega) \subseteq L_\omega$ is equivalent to $\omega^{\flat} \circ r= r^* \circ \omega^{\flat}$, while condition $\mathbb{D}^r_X(\Gamma(L)) \subseteq  \Gamma(L)$ is equivalent to 
$\widetilde{S}^r_\omega =0$, see \eqref{eq:Stilde}. 
So the pair $(\omega,r)$ is compatible  if and only if so is the pair $(L_\omega,\omega)$.
When $\omega$ is closed, $\widetilde{S}^r_\omega= d (\omega_r)$
by Lemma~\ref{lem:Stilde}, so 
$\mathbb{D}^r_X(\Gamma(L)) \subseteq  \Gamma(L)$ if and only if $d(\omega_r) = 0$.
In particular, $(\omega,r)$ is a presymplectic-Nijenhuis structure if and only if $(L_\omega,r)$ is a Dirac-Nijenhuis structure.}
\end{ex}

Recall that any  involutive distribution $F\subseteq TM$ defines a Dirac structure $F \oplus \mathrm{Ann}(F)\subseteq \T M$.
The next example illustrates Dirac-Nijenhuis structures of this type (cf. \cite[Sec.~6]{LB2}).

\begin{ex}\label{ex:im(r)}\em
Let $r$ be a Njenhuis operator of constant rank, and consider the distribution $F=\mathrm{Im}(r)\subseteq TM$.
The Nijenhuis property implies that $F$ is involutive, so $L_r:= F \oplus \mathrm{Ann}(F)$ is a Dirac structure.
Since $\Ann(\mathrm{Im}(r)) = \ker(r^*)$, it is simple to check that $(r,r^*)(L_r) \subseteq L_r$. Condition (2) in Definition~\ref{defDN} follows from the identities \ref{Nij_D} and \eqref{eq:dualNij}, so $(L_r,r)$ is Dirac-Nijenhuis.
\end{ex}

\begin{rmk}[Quasi-Nijenhuis]\label{rem:quasi}
\em There are variants of the definition of Dirac-Nijenhuis structure obtained by weakening the integrability conditions on $L$ or $r$. For example, a {\em Dirac quasi-Nijenhuis} structure is a triple $(L,r,\phi)$, where $L$ is a Dirac structure compatible with $r\in \Omega^1(M,TM)$, $\phi\in \Omega^3(M)$ is a closed 3-form,
and
\begin{equation}\label{eq:quasi}
\<\alpha, \N_r(Y,Z)\> = - \phi(X,Y,Z), \qquad \forall\,  (X,\alpha)\in L.
\end{equation}
Notice that even when $\phi=0$ this condition is weaker than the vanishing of the Nijenhuis torsion $\N_r$: it just says that the image of $\N_r$ lies in $L\cap TM = \mathrm{Ann}(\mathrm{pr}_{T^*}(L))$. When $L$ is the graph of a Poisson structure $\pi$, then \eqref{eq:quasi} becomes
$$
\N_r(X,Y) = \pi^\sharp(\phi(X,Y,\cdot)),
$$
hence recovering the Poisson quasi-Nijenhuis structures of \cite[Def.~3.3]{SX} 
\hfill $\diamond$ 
\end{rmk}

\subsection{Holomorphic Dirac structures}\label{subsc:holDirac}
Suppose that $r\in \Omega^1(M,TM)$ is a complex structure on $M$. It is known (see e.g. \cite[Sec.~3]{LMX2}) that holomorphic Poisson structures on the complex manifold $(M,r)$ are characterized by the fact that their real parts are Poisson-Nijenhuis structures with respect to $r$. We will now extend this result to Dirac structures.

On a complex manifold $(M,r)$, we consider the holomorphic vector bundle $\T^{1,0}M= T^{1,0}M\oplus (T^{1,0}M)^*$ equipped with its  natural \textit{holomorphic Courant algebroid} structure, consisting of the non-degenerate symmetric bilinear pairing $\<\cdot,\cdot\>$ on $\T^{1,0}M$ (obtained from the duality of $T^{1,0}M$ and $(T^{1,0}M)^*$), the projection $\T^{1,0}M \to T^{1,0}M$ and the $\C$-bilinear bracket on the sheaf of holomorphic sections $\Gamma_{\mathrm{ hol}}(\cdot, \T^{1,0}M)$ given by
$$
\Cour{(\mathsf{X},\xi), (\mathsf{Y}, \zeta)}=([\mathsf{X},\mathsf{Y}], \mathcal{L}_{\mathsf{X}}\zeta -\textit{i}_{\mathsf{Y}}\partial\xi),
$$
where $(\mathsf{X},\xi), (\mathsf{Y}, \zeta)$ are local holomorphic sections of $\T^{1,0}M$.
 A \textit{holomorphic Dirac structure on $M$} (see e.g. \cite{Gual14}) is a holomorphic subbundle $\mathcal{L} \subset \T^{1,0}M$ which is lagrangian with respect to $\<\cdot, \cdot\>$ and whose sheaf of holomorphic sections is involutive with respect to $\Cour{\cdot,\cdot}$.

To explain the sense in which holomorphic Dirac structures are equivalent to Dirac-Nijenhuis structures, note that, from Example~\ref{ex:holtancot}, the 
holomorphic structure on $\T M$ coming from its identification with $\T^{1,0} M$ via 
\begin{equation}\label{dfn:phi_iso}
\Phi: \T M \to \T^{1,0} M,\qquad
    \Phi(X,\alpha) = \Big ( \frac{1}{2}(X-\ii r(X)), \alpha - \ii r^*(\alpha) \Big )
\end{equation}
is given by the Dolbeault 1-derivation $(\mathbb{D}^r, (r,r^*), r)$.
Another key property of the map $\Phi$ is that it preserves Courant brackets of holomorphic sections.

\begin{lem}\label{lem:PhiCour}
For (local) holomorphic sections $(X,\alpha), \, (Y, \beta)$ of $\T M$, we have
$$
\Phi(\Cour{(X,\alpha), (Y,\beta)}) = \Cour{\Phi(X,\alpha), \Phi(Y,\beta)}.
$$
\end{lem}

\begin{proof}
Recall that $\Cour{\Phi(X,\alpha), \Phi(Y,\beta)}$ has $T^{1,0}M$-component given by
\begin{equation}\label{eq:compT}
\frac{1}{4}[X-\ii r(X),Y-\ii r(Y)]
\end{equation}
and $(T^{1,0}M)^*$-component given by
\begin{equation}\label{eq:compT*}
\frac{1}{2} (\mathcal{L}_{X-\ii r(X)} (\beta-\ii r^*\beta) - i_{Y-\ii r(Y)}\, \partial(\alpha-\ii r^* \alpha)).
\end{equation}
On the other hand,
$$
\Phi(\Cour{(X,\alpha),(Y,\beta)}) =\left(\frac{1}{2}([X,Y]- \ii r[X,Y]), \mathcal{L}_X\beta -i_Yd\alpha -\ii r^*(\mathcal{L}_X\beta -i_Yd\alpha) \right)
$$
Recall (see \eqref{eq:nablaD}) that $(X,\alpha)$ is holomorphic if and only if $\D^r_Z(X,\alpha) = 0$, for all $Z \in \frakx(M)$, and this condition implies the following identities for $(X,\alpha)$ (similarly for $(Y, \beta)$):
\begin{align*}
&[X, r(Z)] = r([X,Z]),  \qquad \Lie_X r^*(\gamma) = r^*(\Lie_X\gamma),\\
& \Lie_{r(Z)} \alpha = \Lie_Z r^*(\alpha),  \qquad \;\;\;\; i_{r(Z)} d\alpha = i_Z d(r^*(\alpha)), 
\end{align*}
for all $Z \in \frakx(M)$, and $\gamma \in \Omega^1(M)$.
It follows that \eqref{eq:compT} equals
$$
\frac{1}{4}([X,Y]-[r(X),r(Y)]- \textbf{i}([r(X),Y]+[X,r(Y)]))\\
 =  \frac{1}{2}([X,Y]- \textbf{i}r([X,Y])).
$$
Similarly, \eqref{eq:compT*} equals (recall that $d=\partial$ on holomorphic forms)
\begin{align*}
& \frac{1}{2}\left(\Lie_X\beta - \Lie_{r(X)}r^*(\beta) - {\ii}(\Lie_{r(X)}\beta + \Lie_Xr^*(\beta)) -  2(i_Y d\alpha - \ii i_{r(Y)} d\alpha)\right)\\
& = \Lie_X \beta - i_Y d\alpha - \ii r^*(\Lie_X \beta - i_Y d\alpha).
\end{align*}
\end{proof}

We now verify that holomorphic Dirac structures correspond to Dirac-Nijenhuis structures under the isomorphism \eqref{dfn:phi_iso}.

\begin{prop}\label{prop:hol_dirac}
Let $r\in \Omega^1(M,TM)$ be a complex structure on $M$ and $L \subset \T M$ be a lagrangian subbundle. The pair $(L,r)$ is compatible if and only if  $\Phi(L)$ is a holomorphic lagrangian subbundle of $\T^{1,0}M$, and $(L,r)$ is a Dirac-Nijenhuis structure if and only if $\Phi(L)$ is a holomorphic  Dirac structure.
\end{prop}

\begin{proof}
A real subbundle $L\subset \T M$ is a complex subbundle if and only if $(r,r^*)(L)\subseteq L$, and such a subbundle is holomorphic if and only if $\mathbb{D}^r_X(\Gamma(L))\subseteq \Gamma(L)$ for all $X\in \mathfrak{X}(M)$. Also, for a complex subbundle $L\subseteq \T M$ 
one can directly check that $\Phi(L)^{\perp} = \Phi(L^\perp)$, so $L$ is lagrangian in $\T M$ if and only if $\Phi(L)$ is lagrangian in $\T^{1,0}M$.  This proves the first assertion of the proposition.

For the second assertion, we note that the involutivity of $L$ is equivalent to the vanishing of $T \in \Gamma(\wedge^3 L^*)$ given by
$$
T(\sigma_1,\sigma_2, \sigma_3) := \<\Cour{\sigma_1,\sigma_2}, \sigma_3\>.
$$
For holomorphic sections $\sigma_1, \sigma_2, \sigma_3 \in \Gamma(L)$, we have
\begin{align*}
    \<\Cour{\Phi(\sigma_1),\Phi(\sigma_2)}, \Phi(\sigma_3)\>&= 
    \<\Phi(\Cour{\sigma_1, \sigma_2}), \Phi(\sigma_3)\>
     = T(\sigma_1, \sigma_2, \sigma_3) - \ii T(\sigma_1, \sigma_2,(r,r^*)(\sigma_3)).
\end{align*}
The condition $\Phi(L) = \Phi(L)^\perp$ implies  that the sheaf of holomorphic sections of $\Phi(L)$ is involutive if and only if $T$ vanishes on holomorphic sections of $L$, which is in turn equivalent to the vanishing of $T$
(since any smooth section $\sigma$ of $L$ can be written locally as $\sigma=\sum_i f_i \sigma_i$,  where $f_i \in {C}^{\infty}(M)$ and $\sigma_i$ is holomorphic).
\end{proof}

For a complex lagrangian subbundle $\mathcal{L}\subset \T^{1,0} M$, we define its \textit{real part} as the real lagrangian subbundle $L\subset \T M$ such that $\Phi(L) = \mathcal{L}$.

\begin{ex}\label{ex:poisson_hol}\em
Let $\Pi = \pi + \ii \pi_1 \in \Gamma(\wedge^2 TM \otimes \C)$ be a complex bivector field. Then $\Pi \in \Gamma(\wedge^{2}T^{1,0}M)$ if and only if $\pi_1^\sharp = -\pi^\sharp \circ r^*$, i.e., $\pi_1=-\pi_r$. In this case, the lagrangian subbundle $\mathcal{L}$ of $\T^{1,0}M$ defined as the graph of $\Pi^\sharp: (T^{1,0}M)^* \to T^{1,0}M$ has real part given by $L=\mathrm{graph}(4\pi^\sharp)$. By Proposition \ref{prop:hol_dirac} , we have a bijective correspondence between Poisson-Nijenhuis structures $(\pi, r)$ and holomorphic Poisson structure $\Pi = \pi - \ii \pi_r$, see \cite[Prop.~2.6]{LMX2}.
\end{ex}

\begin{ex}\label{ex:presymphol}\em 
Let $\Omega = \omega + \ii \omega_1 \in \Gamma(\wedge (TM \otimes \C)$ be a complex 2-form. Then $\Omega \in \Gamma(\wedge^2 (T^{1,0}M)^*)$ if and only if $\omega_1 = - \omega_r$, and the real part of the lagrangian subbundle $\mathcal{L}:=\mathrm{graph}(\Omega^\flat)\subset \T^{1,0}M$, where $\Omega^\flat: T^{1,0}M\to (T^{1,0}M)^*$, is given by $L=\mathrm{graph}(\omega^\flat)$.  Proposition \ref{prop:hol_dirac} thens gives a bijective correspondence between presymplectic-Nijenhuis structures $(\omega, r) $ and holomorphic closed 2-forms $\Omega = \omega - \ii \omega_r$.

\end{ex}

We will generalize Example \ref{ex:presymphol} to higher-degree forms in Proposition \ref{prop:hol_dif} below.

\section{Properties of Dirac-Nijenhuis structures}
\label{sec:properties}
In this section we discuss presymplectic foliations and Poisson quoients of Dirac-Nijenhuis structures, 
 as well as hierarchies that extend those of Poisson-Nijenhuis structures.

\subsection{Backward maps and presymplectic leaves}

Let $N$ be a smooth manifold, and let $\psi: N \to M$ be a smooth map. If $L_N$, $L_M$ are Dirac structures on $N$ and $M$, respectively, recall that $\psi$ is a \textit{backward Dirac map} (see e.g. \cite{HB}) if 
$$
(L_N)_x = \{(X, (T_x\psi)^*(\beta)) \in \T_x N \,\, | \,\, (T_x\psi(X),\beta) \in (L_M)_{\psi(x)}\}.
$$
Recall that $r_N\in \Omega^1(N,TN)$ and $r_M\in \Omega^1(M,TM)$ are {\em $\psi$-related} if $r_M \circ d\psi = d\psi \circ r_N$.

\begin{prop}\label{prop:backdirac}
Suppose that $\psi:(N,L_N) \to (M, L_M)$ is a backward Dirac map
and that $r_N$ and $r_M$ are $\psi$-related.
If $L_M$ is compatible with $r_M$, then $L_N$ is compatible with $r_N$.
\end{prop}

\begin{proof}
The fact that $r_N$ and $r_M$ are $\psi$-related directly implies that  $(r_N, r_N^*)(L_N) \subset L_N$.
To show that $L_N$ verifies condition (2) in Definition~\ref{defDN}, we can equivalently verify the vanishing of the  concomitant $C^{r_N}_{L_N} \in \Gamma(\wedge^2 L_N^* \otimes T^*N)$ of Remark~\ref{rem:concom}.

We say that (local) sections $a=(X,\alpha)$ of $\T N$ and $b=(Y,\beta)$ of $\T M$ are $\psi$-related if $X$ and $Y$ are $\psi$-related, and $\alpha=\psi^*\beta$. Note that if $a\in \Gamma(\T N)$ is $\psi$-related to $b\in \Gamma(L_M)$, then it follows from the definition of backward Dirac map that $a\in \Gamma(L_N)$.  

Suppose that $a\in \Gamma(L_N)$ is $\psi$-related to $b\in \Gamma(L_M)$, and let $Z_1\in\mathfrak{X}(N)$ and $Z_2 \in \mathfrak{X}(M)$ be $\psi$-related.
It directly follows from the naturality properties \eqref{eq:natural1} and \eqref{eq:natural2} that $\D_{Z_1}^{r_N}(a)$ is $\psi$-related to $\D_{Z_2}^{r_M}(b) \in \Gamma(L_M)$ (since we are assuming that $L_M$ is $r_M$-compatible), and hence  
$\D_{Z_1}^{r_N}(a)\in \Gamma(L_N)$. Now suppose that $x\in N$ satisfies the following property: any $a_x\in L_N|_x$ and $(Z_1)_x\in TN|_x$ can be extended to local sections $a$ of $L_N$ and $Z_1$ of $TN$ which are $\psi$-related to some
local sections $b$ of $L_M$ and $Z_2$ on $TM$, respectively; in this case, $C^{r_N}_{L_N} |_x =0$. We finally observe that the points $x\in N$ with this property form an open dense subset of $N$ (see the proof of Prop.~5.6 in \cite{HB}), hence $C^{r_N}_{L_N} = 0$.
\end{proof}

On a Dirac manifold $(M, L)$, let $\psi: C \hookrightarrow M$ be a submanifold. It is known (see e.g. \cite[Sec.~5]{HB}) that if the lagrangian distribution 
\begin{equation}\label{eq:back}
\mathcal{B}_\psi(L)=\{(X,\psi^*\alpha) \in \mathbb{T}C \: | \:  T\psi(X), \alpha) \in L\} \subset TC\oplus T^*C
\end{equation} 
defined by pointwise pullback of $L$ to $C$ fits into a smooth vector bundle  (this is ensured by suitable clean-intersection conditions \cite{Cour}, see also \cite{HB}), then it defines a Dirac structure on $C$ for which $\psi$ is a backward Dirac map. 

For $r\in \Omega^1(M,TM)$, we call a submanifold $C\hookrightarrow M$ {\em $r$-invariant} if $r(TC)\subseteq TC$; we define $r_C \in \Omega^1(C,TC)$ by $r|_{TC}$.

\begin{cor}\label{cor:DNsubm}
Let $L$ be a Dirac structure on $M$ compatible with $r\in \Omega^1(M,TM)$. Let $\psi: C\hookrightarrow M$ be an $r$-invariant submanifold for which  $\mathcal{B}_\psi(L)$ in \eqref{eq:back} is smooth, hence a Dirac structure on $C$. Then $\mathcal{B}_\psi(L)$ is compatible with $r_C$.
\end{cor}

Note that $r_C$ is Nijenhuis if $r$ is Nijenhuis. So the previous corollary illustrates how Dirac-Nijenhuis structures naturally arise on submanifolds of Dirac-Nijenhuis manifolds. 

A special case is that of presymplectic leaves of a Dirac manifold $(M,L)$. Recall that $L$ defines a (singular) foliation on $M$ whose leaves are tangent to the distribution given by the projection of $L$ on $TM$; moreover, each leaf $\O$ carries a closed $2$-form $\omega_\O \in \Omega^2(\mathcal{O})$, which can be seen as  the pullback of $L$ to $\O$. Note also that the condition $(r,r^*)(L) \subset L$ implies that $T\O = \pr_T(L)$ is preserved by $r$. 

\begin{cor}
Let $L$ be a Dirac structure on $M$ compatible with $r\in \Omega^1(M,TM)$. Then each leaf $\O$ is $r$-invariant, and its presymplectic form $\omega_\O$ is compatible with the restriction $r_\O$. In particular, if $(M,L,r)$ is a Dirac-Nijenhuis manifold, then its leaves are presymplectic-Nijenhuis manifolds.
\end{cor}


\subsection{Forward maps and Poisson quotients} 
A map $\psi: (M, L_M) \to (N,L_N)$ between Dirac manifolds is a \textit{forward Dirac map} if 
$$
(L_N)_{\psi(x)} = \{(T_x\psi(X), \beta) \in \T_{\psi(x)}N \,\, | \,\, (X, (T_x\psi)^*\beta) \in (L_M)_x\}.
$$
As in the proof of Proposition~\ref{prop:backdirac}, one can use the naturality properties of Remark~\ref{rem:naturality} to
show the following result: 

\begin{prop}\label{prop:forward_dirac}
Let $\psi: (M, L_M) \to (N, L_N)$ be a surjective submersion that is a  forward Dirac map.  Suppose that $r_M\in \Omega^1(M,TM)$ is $\psi$-related to $r_N\in \Omega^1(N,TN)$. Then $L_N$ is compatible with $r_N$ provided $L_M$ is compatible with $r_M$.
\end{prop}

Note that if $r_M$ is Nijenhuis, then so is $r_N$, hence the previous result illustrates how Dirac-Nijenhuis structures may arise on quotients. The following is a special case.

Recall \cite{Cour} (see also \cite{HB}) that the \textit{null distribution} of a Dirac structure $L$ on $M$ is 
$$
K:=L \cap TM \subseteq TM.
$$ 
At each point $x\in M$, $K$ agrees with the kernel of the  presymplectic form on the leaf through $x$. Whenever $K$ has constant rank, it is integrable and defines the so-called {\em null foliation} $\mathcal{K}$. If the null foliation is simple, then $Q:=M/\mathcal{K}$ inherits a Poisson structure $\pi_Q$ uniquely determined by the fact that the quotient map $M\to Q$ is a forward Dirac map. 

Given $r\in \Omega^1(M,TM)$ and a surjective submersion $\psi: M\to N$, we say that $r$ {\em descends} to $N$ if
 there exists $r_N\in \Omega^1(Q,TN)$ that is $\psi$-related to $r$, in which case $r_N$ is uniquely determined by this property.

\begin{cor}\label{cor:PNquot}
Let $L$ be a Dirac structure on $M$ whose null foliation $\mathcal{K}$ is simple, let $Q$ be the leaf space with induced Poisson structure $\pi_Q$, and suppose that $L$ is compatible with $r\in \Omega^1(M,TM)$. Then $r$ descends to $r_Q\in \Omega^1(Q,TQ)$ which is compatible with $\pi_Q$.
In particular, if $(M,L,r)$ is a Dirac-Nijenhuis manifold, its quotient $Q=M/\mathcal{K}$ is a Poisson-Nijenhuis manifold.
\end{cor}

\begin{proof}
Once we prove that $r$ descends to $Q$, the result follows directly from Proposition \ref{prop:forward_dirac}. 

The condition $(r,r^*)(L) \subset L$ implies that $r(K)\subseteq K$ and $r^*(\mathrm{Ann}(K))\subseteq \mathrm{Ann}(K)$  (since $\mathrm{Ann}(K)=\pr_{T^*}(L)$). To verify that $r$ descends to $Q$, we must check that $r$ preserves projectable vector fields, i.e., $r(\mathfrak{X}(M)_{pr})\subseteq \mathfrak{X}(M)_{pr}$,
where 
$$
\mathfrak{X}(M)_{pr}=
\{Y \in \frakx(M)\,\, | \,\, [X, Y] \in \Gamma(K), \,\, \forall X \in \Gamma(K)\}.
$$
So let $\alpha \in \Gamma(\mathrm{Ann}(K))$, $Y \in \frakx(M)_{pr}$ and $X \in \Gamma(K)$, and note that
\begin{align*}
\<\alpha, [X, r(Y)]\> & = i_X \Lie_{r(Y)}\alpha - \cancel{\Lie_{r(Y)}i_X\alpha}^{=0} = i_X(-D^{r,*}_Y(\alpha) + \Lie_Y(r^*\alpha))\\
& = -\cancel{\<D^{r,*}_Y(\alpha), X\>}^{=0} + \cancel{i_{[X,Y]}r^*(\alpha)}^{=0} + \cancel{\Lie_Y i_X r^*(\alpha)}^{=0} = 0,
\end{align*}
where for the last equality, besides the fact that $r^*$ preserves $\mathrm{Ann}(K)$,  we used that $D^{r,*}_Y$ preserves $\Gamma(\mathrm{Ann}(K))$ (which follows from $\mathbb{D}^r_Y(\Gamma(L))\subseteq \Gamma(L)$ and $\mathrm{Ann}(K)=\pr_{T^*}(L)$). 
\end{proof}

\subsection{Traces of powers of $r$.}\label{subsec:traces}
An important property of a PN manifold $(\pi,r)$ is that 
\begin{equation}\label{dfn:traces}
    \phi_j := \frac{1}{j}\mathrm{trace}(r^j), \,\,\, j \geq 1,
\end{equation}
are functions in involution:  $\{\phi_i,\phi_j\}=0$, $\forall \, i,\,j \geq 1$, see \cite{MM}. We now present a generalization of this result for Dirac-Nijenhuis manifolds.

Recall that, for a Dirac structure $L$  on $M$, the space of \textit{admissible functions} on $M$, denoted by $C^\infty_{adm}(M)$, is the space of functions $f \in C^{\infty}(M)$ for which there exists $X_f \in \frakx(M)$ satisfying $(X_f,df) \in \Gamma(L)$. Although such $X_f$ may not be uniquely defined, there is a Poisson bracket on $C^\infty_{adm}(M)$ defined by the usual formula 
$$
\{f,g\} = dg(X_f),
$$
see \cite{Cour}. In case the the null distribution is regular, a function is admissible if and only if it is constant along the leaves of the null foliation $\mathcal{K}$. If $\mathcal{K}$ is simple, the natural identification $ C^{\infty}(M/\mathcal{K}) \cong C^{\infty}_{adm}(M)$ is an isomorphism of Poisson algebras.

\begin{prop}
Let $(L,r)$ be a Dirac-Nijenhuis manifold. If $\phi_1 = \mathrm{trace}(r)$ is admissible, then $\phi_i$ is admissible for every $i\geq 2$ and 
$$
\{\phi_i,\phi_j\}=0, \, \, \forall \, i, j \geq 1.
$$
\end{prop}
\begin{proof}
Recall the useful formula $r^*(d\phi_i) = d\phi_{i+1}$ (see \cite[Eq.~4.2]{MM}). So, if $\phi_1$ is admissible, there exists $X \in \frakx(M)$ such that $(X,d\phi_1) \in \Gamma(L)$. In this case, since $L$ and $r$ are compatible,
$$
\Gamma(L) \ni (r,r^*)^{i-1}(X, \phi_1) = (r^{i-1}(X), d\phi_i) \Longrightarrow \text{ $\phi_i$ is admissible. }
$$
Also, a simple calculation in local coordinates shows that 
$
\<d\,\mathrm{trace}(r), X\> = \mathrm{trace}(D^r_{(\cdot)}(X)).
$
Now, using that $\mathrm{trace}(AB)=\mathrm{trace}(BA)$ and
$$
D^{r^j}_X(Y) = \sum_{k=0}^j r^{j-k}(D^r_{r^k(X)}(Y)),
$$
one shows that 
\begin{align*}
\{\phi_i,\phi_j\} & = d\phi_j(r^{i-1}(X)) = \mathrm{trace}(r^j(D_\cdot^r(r^{i-1}(X)))),\\
\{\phi_j,\phi_i\} & = d\phi_i(r^{j-1}(X)) = \mathrm{trace}(r^i(D_\cdot^r(r^{j-1}(X)))).
\end{align*}
As $D^r(r(X))=r(D^r(X))$ (see \eqref{Nij_D}), one obtains that $\{\phi_i,\phi_j\}=\{\phi_j,\phi_i\}$, which can only happen if $\{\phi_i,\phi_j\}=0$.
\end{proof}

The following simple example shows that the admissibility condition on $\mathrm{trace}(r)$ is not automatic.
 Moreover, even when it holds and the Poisson-Nijenhuis quotient $Q=M/\mathcal{K}$ is well defined (see Cor.~\ref{cor:PNquot}), we will see that the relationship between $\mathrm{trace}(r)$ and the  pull-back of $\mathrm{trace}(r_Q)$ is not straightforward.

\begin{ex}\em
In $\R^2$ with coordinates $(x_1,x_2)$, consider  the distribution $F$ generated by $\frac{\partial}{\partial x_2}$. Consider $L=F \oplus \mathrm{Ann}(F)$ and 
\begin{align*}
r & = a(x_1)\frac{\partial}{\partial x_1}\otimes dx_1 + b(x_2) \frac{\partial}{\partial x_2} \otimes dx_2\\
\widetilde{r} & = a(x_1)\frac{\partial}{\partial x_1}\otimes dx_1 + c(x_1) \frac{\partial}{\partial x_2} \otimes dx_2
\end{align*}
for $a,b,c: \R \to \R$ smooth functions. One can check that $(L,r)$ is Dirac-Nijenhuis, but $\mathrm{trace}(r)$ is admissible if and only if $b$ is a constant function. 

On the other, the pair $(L,\widetilde{r})$ is compatible, $\mathrm{trace}(\widetilde{r})$ is admissible, but $\widetilde{r}$ is Nijenhuis if and only if 
\begin{equation}\label{eq:Nij_2dim}
(a-c) \frac{dc}{dx_1} =0.
\end{equation}
Then the relationship between $\mathrm{trace}(\widetilde{r})=a+c$ and $\mathrm{trace}(\widetilde{r}_Q)=a$ is  that    
$$
\mathrm{trace}(\widetilde{r}) = 2 \mathrm{trace}(\widetilde{r}_Q) \;\; \mbox{ or }\;\; \mathrm{trace}(\widetilde{r}) =\mathrm{trace}(\widetilde{r}_Q) + \lambda, 
$$
for $\lambda \in \R$, depending on whether $a=c$ or $dc/dx_1 = 0$, respectively. But it is possible to have a mixture of these situations if there exist points $\bar{x}_1$ where $d^na/dx_1^n= 0$ for all $n \geq 1$. In this case, $c$ defined as $a$ on $(-\infty, \bar{x}_1)$ and as $\lambda=a(\bar{x}_1)$ on $[\bar{x}_1, +\infty)$ is smooth and satisfies \eqref{eq:Nij_2dim}.
\end{ex}

\subsection{Hierarchies of Dirac-Nijenhuis structures}\label{subsec:hierarchy}
A fundamental aspect in the theory of Poisson-Nijenhuis structures  $(\pi, r)$ is the existence of a hierarchy of Poisson-Nijenhuis structures $(\pi_n,r)$, $n=1,2,\ldots$, recursively defined by
\begin{align}
\label{eq:pi_recursion}  \pi_0:=\pi, \qquad &  \pi_{n}^\sharp=r \circ \pi_{n-1}^\sharp = r^n\circ\pi^\sharp,
\end{align}
see e.g. \cite{MM,KS-Magri90}.
This hierarchy is a key ingredient in the construction of integrals for bihamiltonian systems admitting a Poisson-Nijenhuis description. Similarly, a presymplectic-Nijenhuis structure $(\omega,r)$ gives rise to a hierarchy of presymplectic-Nijenhuis structures $(\omega_n,r)$, $n=1,2,\ldots$, where
\begin{equation}\label{eq:omegahier}
    \omega_n^\flat = \omega^\flat \circ r^n.
\end{equation}
In particular, symplectic-Nijenhuis structures give rise to two hierarchies, one of Poisson structures, and the other of presymplectic forms. We now see how to extend these hierarchies to Dirac-Nijenhuis structures.

Let $(L,r)$ be a Dirac-Nijenhuis structure on $M$. For each $n \in \mathbb{N}$, let 
\begin{equation}
\label{eq:hierarchy}
L_{(n,0)}:=(r^n, \mathrm{id}_{T^*M})(L),\qquad L_{(0,n)}:=(\mathrm{id}_{TM},(r^*)^n)(L).
\end{equation}

\begin{prop}\label{prop:DNhier}  
(1) If $\ker(r,\mathrm{id}_{T^*M})|_{L}=0$, then $(L_{(n,0)}, r)$ is a Dirac-Nijenhuis structure for all $n\in \mathbb{N}$; (2) if $ \ker(\mathrm{id}_{TM}, r^*)|_L =0$, then $(L_{(0,n)},r)$ is a Dirac-Nijenhuis structure for all $n\in \mathbb{N}$.
\end{prop}

\begin{proof}
We will prove (1); the proof of (2) is similar.

Suppose that $\ker(r,\mathrm{id}_{T^*M})|_L=0$. In this case, $L_{(1,0)}$ is a smooth subbundle of $\mathbb{T}M$ with the same rank as $L$. We will prove that $L_{(1,0)}$ is isotropic (hence lagrangian) and involutive with respect the Courant bracket. The result for a general $L_{(n,0)}$ will follow by induction. 

To prove that $L_{(1,0)}$ is isotropic, note that, for
 $(X,\alpha), (Y,\beta) \in \Gamma(L)$,
   $$
   \< (r(X),\alpha), (r(Y),\beta)\> = \<(X,\alpha), (r, r^*)(Y,\beta)\>=0
   $$
since $(r,r^*)(L) \subset L$.  To prove that involutivity of $L_{(1,0)}$,  we will show that
\begin{equation*}
    \<\Cour{(r(X),\alpha),(r(Y),\beta)}, (r(Z),\gamma)\>=0 
\end{equation*}
for any $(X,\alpha),(Y,\beta),(Z,\gamma) \in \Gamma(L) $. From the definitions of $D^r$ and $D^{r,*}$, see \eqref{eq:Dr} and \eqref{eq:D*}, one can check that 
\begin{align}
\label{eq:involutivity}\<\Cour{(r(X),\alpha),(r(Y),\beta)}, (r(Z),\gamma)\>  
  & = \<\,\Cour{(r(X),r^*(\alpha)), (Y,\beta)}, (r(Z), r^*(\gamma))\,\>\\
\nonumber   & \hspace{-70pt} + \<\D^r_Y(X,\alpha), (r(Z),r^*(\gamma))\> + \<\gamma, \N_r(X,Y)\>.
\end{align}
Since $(r,r^*)(L) \subset L$, $\D^r_Y(\Gamma(L)) \subset \Gamma(L)$ and $\N_r = 0$, we conclude that $L_{(1,0)}$ is a Dirac structure on $M$. The fact that $L_{(1,0)}$ and $r$ are compatible is a direct consequence of \eqref{Nij_D}. For the inductive step, notice that, if $L_{(n-1,0)}$ is a Dirac structure, then
$$
\ker(r^n,\mathrm{id}_{T^*M})|_L = \ker(r, \mathrm{id}_{T^*M})|_{L_{(n-1,0)}} \,\,\, \text{ and } L_{(n,0)} = (r, \mathrm{id}_{T^*M})(L_{(n-1,0)}).
$$
\end{proof}

\begin{rmk}\label{rem:involutive}\em
For a Dirac structure $L$ compatible with a (1,1)-tensor field $r$, it follows from \eqref{eq:involutivity} that a weaker condition on $\N_r$ is sufficient to ensure the involutivity of $L_{(1,0)}$. Indeed, one only needs that 
$$
\N_r(X,Y) \in L \cap TM, \;\;\; \forall \, X, Y \in \pr_T(L).
$$
(In other words., $(L, r, 0)$ is Dirac-quasi Nijenhuis, see Remark~\ref{rem:quasi}). On the other hand, $L_{(0,1)}$ is always involutive as a consequence of the identity
$$
\<\Cour{(X,r^*(\alpha)),(Y,r^*(\beta))}, (Z,r^*(\gamma))\> = \<\Cour{(X,\alpha), (Y,\beta)}, (r(Z), r^*(\gamma))\> + \<\D^r_Z(X,\alpha), (Y,\beta)\>.
$$
\hfill $\diamond$
\end{rmk}

\begin{rmk}\label{rem:invert}\em
The conditions appearing in Proposition~\ref{prop:DNhier} admit the following natural interpretations.
For a Dirac-Nijenhuis structure $(L,r)$, recall that $r$ preserves $K=L\cap TM$ and $r^*$ preserves $L \cap T^*M$. One can check that
\begin{itemize}
\item[(a)] $\ker(r, \mathrm{id}_{T^*M})|_L = 0  \, \Leftrightarrow\, r|_{K}: K \to K \text{ is invertible}$,
\item[(b)] $\ker(\mathrm{id}_{TM}, r^*)|_L = 0  \, \Leftrightarrow\, r^*|_{L\cap T^*M}: L \cap T^*M \to L\cap T^*M \text{ is invertible}.$
\end{itemize}
Since the leaves $\O$ of $L$ are $r$-invariant,
we have an induced map $\bar{r}: TM/T\O \to TM/T\O$.
Hence (b) above is also equivalent to the invertibility of $\bar{r}$ (since $r^*|_{L\cap T^*M}=\bar{r}^*$ with the identification
$L\cap T^*M = \mathrm{Ann}(T\O) = (TM/T\O)^*$).
\hfill $\diamond$
\end{rmk}

In the context of Proposition~\ref{prop:DNhier}, the following properties of $L_{(n,0)}$ and $L_{(0,n)}$ are simple to check.

\begin{prop}\label{prop:hierprop}
\begin{itemize}
    \item[(1)] The Dirac structures $L_{(n,0)}$ have the same null distribution as $L$. If the null foliation is simple and $Q$ is the leaf space, then 
\begin{equation}\label{eq:poisrel}
(\pi_{Q})_n^\sharp = r_Q^n \circ \pi_Q^\sharp,
\end{equation} 
where $(\pi_{Q})_n$ is the Poisson structure on $Q$ corresponding to $L_{(n,0)}$.
     \item[(2)]  The Dirac structures $L_{(0,n)}$ have the same leaves as $L$, and we have the relation
\begin{equation}\label{eq:presymrel}
(\omega_{\O})_n^\flat = \omega_\O^\flat \circ r_\O^n,
\end{equation}
where $(\omega_\O)_n$ is the presymplectic form  defined by $L_{(0,n)}$ on a leaf $\O$.
\end{itemize}
\end{prop}

\begin{ex}[Poisson hierarchies]\em \label{exPN} Let $(M, \pi, r)$ be a Poisson-Nijenhuis manifold, and let $L=\mathrm{graph}(\pi^{\sharp})$. Since $L\cap TM=0$, the hierarchy $L_{(n,0)}$ is always well defined and is given by 
$$
L_{(n,0)} = \mathrm{graph}(\pi_n^{\sharp}),
$$
where the Poisson structure $\pi_n$ is defined by \eqref{eq:pi_recursion}. Regarding the other hierarchy, the condition $\ker(\mathrm{id}_{TM}, (r^*)^n)|_L = 0$ is equivalent to $\ker(\pi^\sharp) \cap \ker((r^*)^n) = 0$. If this holds, as we saw in Prop.~\ref{prop:hierprop}, part (2), $L_{(0,n)}$ is a Dirac structure whose leaves $\mathcal{O}$ agree with the leaves of $\pi$, and the leafwise presymplectic form $(\omega_{\O})_n$ is related to the symplectic form $\omega_\O$ on $\O$ corresponding to $\pi$ by \eqref{eq:presymrel}. Notice that $L_{(0,n)}$ is the graph of a Poisson structure if and only if each $r_\O$ is invertible, which is equivalent to $r$ being invertible by Remark~\ref{rem:invert} (see also Remark~\ref{rem:invertible}).
\end{ex}

The next example shows how gauge transformations of Poisson structures \cite{SW} may be seen as elements of a hierarchy.

\begin{ex}[Gauge transformations]\label{ex:gauge} \em
Let $(M,\pi)$ be a Poisson manifold and $B \in \Omega^2(M)$. Define $r:= \id + \pi^\sharp \circ B^\flat: TM \to TM$, 
noticing that $\pi^\sharp \circ r^* = r \circ \pi^\sharp$. Also,  as shown in \cite{MM} (see equation B.3.9 therein), 
$$
R_\pi^r(X, \alpha) = \pi^\sharp ( i_Xi_{\pi^\sharp(\alpha)}dB).
$$
By assuming that $B$ is closed, one then has that the pair $(\pi, r)$ is compatible and (cf. Remark~\ref{rem:involutive})
$$
L_{(0,1)} = \{(\pi^\sharp(\alpha), \alpha + i_{\pi^\sharp(\alpha)}B),\,\; \, \alpha \in T^*M\}
$$
is a Dirac structure. The Dirac structure $L_{(0,1)}$ is known as the \textbf{gauge transformation} of $\pi$ by $B$. 
Note that $L_{(0,1)}$ is again the graph of a Poisson structure if and only if $r$ is invertible. 
The vanishing of the Nijenhuis torsion of $r$, making the pair $(\pi,r)$ into a Poisson-Nijenhuis structure, is equivalent to the 2-form defined by $B^\flat \circ \pi^\sharp \circ B^\flat: TM \to T^*M$ being closed, which is the same as saying that the pair  $(B, r)$ is presymplectic-Nijenhuis \cite{MM}. 
\end{ex}

\begin{ex}[Presymplectic hierarchies] \em \label{exhierpres}
Let $(M,\omega,r)$ be a presymplectic-Nijenhuis manifold, and let  $L= \mathrm{graph}(\omega^\flat)$. Since $L\cap T^*M =0$, the hierarchy $L_{(0,n)}$ is always well defined and
$$
L_{(0,n)}= \mathrm{graph}(\omega^{\flat} \circ r^n).
$$
As for the other hierarchy, the condition $\ker(r^n, \mathrm{id}_{T^*M})|_L=0$ is equivalent to $\ker(r^n) \cap \ker(\omega^\flat) = 0$. Following Prop.~\ref{prop:hierprop} (1), in this case $L_{(n,0)}$ is a Dirac structure whose null distribution coincides the kernel $K$ of $\omega$.
When $K$ is tangent to a simple foliation with leaf space $Q$, the Poisson structure $(\pi_{Q})_n$ on $Q$ corresponding to $L_{(n,0)}$ satisfies \eqref{eq:poisrel}.
\end{ex}

\begin{rmk}\label{rem:invertible}
{\em   For a Dirac-Nijenhuis structure $(L,r)$ on $M$ with invertible $r: TM\to TM$, both hierarchies $L_{(n,0)}$ and $L_{(0,n)}$ are well defined. In this case one can verify that $(L,r^{-1})$ is also a Dirac-Nijenhuis structure using the relation
$$
D^{r^{-1}}_X(Y) = r^{-1}(D_{r^{-1}(X)}^r(Y)).
$$
Moreover, the ``$(0,n)$'' hierachy of $(L,r)$ coincides with the ``$(n,0)$'' hierarchy of $(L,r^{-1})$. }
\hfill $\diamond$
\end{rmk}

The next example describes the hierarchy associated to a holomorphic Dirac structure. It will be convenient to define the operation on Dirac structures, $L \mapsto L^{-}$, given by 
$$
L^{-} = (-\mathrm{id}_{TM}, \mathrm{id}_{T^*M})(L).
$$

\begin{ex}[Holomorphic hierarchy]\em
Let $(M,r)$ be a complex manifold. Let $\mathcal{L}  \subset \T^{1,0}M$ be a holomorphic Dirac structure, with corresponding Dirac-Nijenhuis structure $(L,r)$, see Prop.~\ref{prop:hol_dirac}. Recall that $L$ is referred to as the real part of $\mathcal{L}$.
One can check that
$$
\xymatrix{L \ar@/^/[r]^{(r,\mathrm{id})} 
& L_{(1,0)} \ar@/^/[l]^{(\mathrm{id},r^*)} \ar@/^/[r]^{(r,\mathrm{id})} 
& L^- \ar@/^/[l]^{(\mathrm{id},r^*)} \ar@/^/[r]^{(r,\mathrm{id})} 
& L_{(1,0)}^- \ar@/^/[l]^{(\mathrm{id},r^*)} \ar@/^/[r]^{(r,\mathrm{id})} 
& L \ar@/^/[l]^{(\mathrm{id},r^*)}
}
$$
We define the {\em imaginary part} of $\mathcal{L}$ as $L_{(0,1)}=L_{(1,0)}^-$. When $\mathcal{L}= \mathrm{graph}(\Pi)$ (resp. $\mathrm{graph}(\Omega)$) for a holomorphic Poisson structure $\Pi=\pi - \ii \pi_r$ (resp. holomorphic closed 2-form $\Omega= \omega - \ii \omega_r$), the imaginary part of $\mathcal{L}$ is $\mathrm{graph}(-4\pi_r)$ (resp. $\mathrm{graph}(-\omega_r)$), see Examples \ref{ex:poisson_hol} and \ref{ex:presymphol}.
\end{ex}

A central fact concerning the Poisson structures $\pi_n$, $n=1,2,\ldots$, in the hierarchy \eqref{eq:pi_recursion} of a Poisson-Nijenhuis structure $(\pi,r)$ is that, for any $i$, $j$ , $\pi_i+\pi_j$ is a Poisson structure; equivalently,  $[\pi_i,\pi_j]=0$, where $[\cdot,\cdot]$ is the Schouten bracket. We briefly explain how this compatibility property can be generalized to Dirac-Nijenhuis hierarchies.

Given Dirac structures $L_1$ and $L_2$ on a manifold $M$, following \cite{Frejlich} we define their ``cotangential product'' as
$$
L_1\circledast L_2 =\{(X_1+ X_2,\alpha)\,|\, (X_1,\alpha)\in L_1,\, (X_2,\alpha) \in L_2  \}.
$$
In general, this product results in a (possibly non-smooth) lagrangian distribution in $\T M$; we say that the Dirac structures {\em concur} if $L_1\circledast L_2$ is a smooth involutive subbundle, i.e., a Dirac structure. For Poisson structures $\pi_1$ and $\pi_2$, it is a direct verification that $L_i=\mathrm{graph}(\pi_i)$, $i=1,2$, concur if and only if $\pi_1+\pi_2$ is a Poisson structure.

For a Dirac-Nijenhuis structure $(L,r)$ on $M$, consider its hierarchy $L_{(n,0)}$, $n=1,2,\ldots$, as in Prop.~\ref{prop:DNhier}, part (1). The fact that all these Dirac structures have the same null distribution $K$ (and hence the same projections on $T^*M$, which are $\mathrm{Ann}(K)$) ensures that their pairwise cotangential products are smooth lagrangian subbundles of $\T M$. Moreover, using a local version of Prop.~\ref{prop:hierprop}, part (1), around  points where $\mathrm{Ann}(K)$ has locally constant rank (which form an open dense subset of $M$) and the fact that the quotient Poisson structures commute, one can verify that any pair $L_{(n,0)}$ and $L_{(m,0)}$ concurs.


\section{Compatibility of multiplicative forms and $(1,1)$-tensor fields}\label{sec:multforms}

In this section we extend the compatibility in Def.~\ref{def:romega} to higher-degree forms.
We refine the infinitesimal description of multiplicative forms in \cite{HA} to include their compatibility with $(1,1)$-tensor fields (Thm.~\ref{thm:compatibility_G}), leading to a  holomorphic version of \cite[Thm.~4.6]{HA} (Thm.~\ref{thm:hol_dif}).

\subsection{Review of multiplicative structures}\label{subsec:multiplicative}
We begin with a brief review of multiplicative differential forms and $(1,1)$-tensor fields on Lie groupoids and their infinitesimal descriptions, following \cite{HA,HT}.

Let $\mathcal{G}\rightrightarrows M$ be a Lie groupoid. We denote its source and targent maps by
$s , t : \G \to M$, and its multiplication map by $m: \G^{(2)} \to \G$, where $\G^{(2)}$ is the fiber product $\mathcal{G}\ttimes \mathcal{G}$. We will often identify $M$ with its image under the unit map. The Lie algebroid of $\G$ is $A_\G:=\ker(Ts)|_M$, with anchor $Tt|_{A_\G}$ and bracket defined by the identification of $\Gamma(A_\G)$ with right-invariant vector fields  on $\G$.

A differential $p$-form $\omega \in \Omega^p(\mathcal{G})$ is called \textit{multiplicative} if 
$$
m^* \omega= \pr_1^*\omega + \pr_2^*\omega,
$$ 
where $\pr_i: \mathcal{G}^{(2)} \rightarrow \mathcal{G}$ is the natural projection, for $i=1,2$. These objects have infinitesimal counterparts given as follows. An {\em IM p-form} on a Lie algebroid $A\to M$ is defined by a pair $(\mu,\nu)$, where 
$$
\mu:A \rightarrow \wedge^{p-1}T^*M, \qquad \nu:A \rightarrow \wedge^p T^*M
$$ 
are vector-bundle maps (over the identity on $M$)
satisfying the so-called \textit{IM equations}:
\begin{equation} \label{eqmubr}
\begin{cases}
     \mu([a,b])  =  \mathcal{L}_{\rho(a)}\mu(b) - i_{\rho(b)}(d\mu(a) + \nu(a))\\ 
\nu([a,b]) =  \mathcal{L}_{\rho(a)}\nu(b) - i_{\rho(b)}d\nu(a)\\
     i_{\rho(a)}\mu(b)  =  -i_{\rho(b)}\mu(a)
\end{cases}
\end{equation}
for all $a,b \in \Gamma(A)$. Any multiplicative $p$-form $\omega$ on a Lie groupoid $\G$ gives rise to an IM $p$-form $(\mu,\nu)$ on the corresponding Lie algebroid $A =  A_\G$ via
\begin{equation}\label{eq:im_form}
    i_{\overrightarrow{a}}\,\omega = t^* \mu(a), \,\;\;\; \;\; i_{\overrightarrow{a}}\,d\omega = t^*\nu(a).
\end{equation}
Here $\overrightarrow{a}$ denotes the right-invariant vector field on $\G$ corresponding to $a \in \Gamma(A)$.
When $\mathcal{G}$ is source 1-connected, \eqref{eq:im_form} defines a one-to-one correspondence between multiplicative $p$-forms on $\G$ and IM $p$-forms on $A$, see \cite{HA}.
Under this bijection, closed multiplicative forms correspond to IM-forms with $\nu=0$, that we will also refer to as {\em closed}.

A $(1,1)$-tensor field $K \in \Gamma(T^*\mathcal{G} \otimes T\mathcal{G})$ is called \textit{multiplicative} if there is a vector bundle map $r:TM \rightarrow TM$ such that
\begin{center}
  \begin{tikzcd}
 
T\mathcal{G} \ar[r, "K"] \ar[d, xshift=0.7ex] \ar[d, xshift=-0.7ex] & T\mathcal{G} \ar[d, xshift=0.7ex] \ar[d, xshift=-0.7ex]  \\
TM \ar[r,"r"] & TM

\end{tikzcd}
\end{center}
is a groupoid morphism. Their infinitesimal counterparts are described as follows \cite[Sec.~6]{HT}. On a Lie algebroid $A$, an {\em IM $(1,1)$-tensor} is a triple 
 $(D,l,r)$, where $l:A \rightarrow A$ and $r: TM \to TM$ are  vector-bundle maps (over the identity on $M$) and $D:\Gamma(A) \rightarrow \Gamma(T^*M \otimes A)$ is a $\mathbb{R}$-linear map satisfying
the Leibniz-type rule
$$
D(fa)=fD(a) + df \wedge l(a) - a \wedge r^*(df), \quad \mbox{ for }\, f \in {C}^{\infty}(M), a \in \Gamma(A),
$$
 and such that the following {\em IM-equations} hold:
\begin{equation}\label{eq:IM_1_1}
\left\{
\begin{array}{rl}   
r \circ \rho \!\! & = \rho \circ l,\\
\rho(D_X(a)) \!\! & = D_X^r(\rho(a)),\vspace{2pt}\\
l([a,b]) \!\! &  = [a,l(b)] - D_{\rho(b)}(a) \vspace{2pt}\\
D_X([a,b]) \!\! &  = [a,D_X(b)] - [b,D_X(a)] + D_{[\rho(b),X]}(a) -D_{[\rho(a),X]}(b).
\end{array}
\right.
\end{equation}
In the language of $\S$~\ref{subsec:1_tensors}, $(D,l,r)$ is a 1-derivation on the vector bundle $A$ satisfying the additional compatibilities \eqref{eq:IM_1_1} with respect to its Lie-algebroid structure.

The IM $(1,1)$-tensor $(D,l,r)$ corresponding to a multiplicative $(1,1)$-tensor field $K$ is defined by the conditions 
\begin{equation}\label{eq:inf_(1,1)}
K(\overrightarrow{a}) = \overrightarrow{l(a)}, \,\,\quad Tt(K(U))=r(Tt(U)), \,\,\quad (\mathcal{L}_{\overrightarrow{a}} K)(U)= \overrightarrow{D_{Tt(U)}(a)},
\end{equation}
for $a \in \Gamma(A)$ and $t$-projectable $U \in \mathfrak{X}(\G)$. With the identification  $T\mathcal{G}|_M=TM \oplus A$,  we directly see that $r=K|_{TM}$ and $l=K|_A.$

As before, when $\mathcal{G}$ is a source 1-connected Lie groupoid, we have a bijective correspondence between
multiplicative $(1,1)$-tensor fields and IM $(1,1)$-tensors on $A$ via \eqref{eq:inf_(1,1)}, see \cite[Thm.~3.19]{HT}.
By regarding vector bundles as Lie groupoids (with respect to fiberwise addition), one can check that multiplicative  $(1,1)$-tensors agree with linear ones, and Theorem~\ref{thm:bijection} follows as a particular case of this result.

For a multiplicative $(1,1)$-tensor field $K$ on a source-connected Lie groupoid $\G \toto M$, it is also shown in \cite[$\S$~6.2]{HT} that $K$ is a Nijenhuis operator if and only if the following equations hold for its corresponding IM $(1,1)$-tensor $(D,l,r)$ (cf. \eqref{eq:IMNij}):
\begin{equation}
\label{eq:im_nijenhuis} 
\begin{cases}
\N_r(X,Y) = 0,\\
l(D_X(a) - D_X(l(a)) = 0,\\
D^2_{(X,Y)}(a):=l(D_{[X, Y]}(a)) - [D_X, D_Y](a) - D_{[X,Y]_r}(a) = 0,
\end{cases}
\end{equation}
for all $X, Y \in \frakx(M)$ and $a \in \Gamma(A)$. Here $[D_X, D_Y]$ is the commutator of the operators on $\Gamma(A)$ and $[\cdot,\cdot]_r$ is the deformed bracket \begin{equation}\label{eq:r_bracket}
[X,Y]_r = [r(X),Y]+[X,r(Y)]-r([X,Y]).
\end{equation}
The notation $D^2$ is motivated by a natural extension of $D$ as an operator on $\Gamma(\wedge^{\bullet} T^*M \otimes A)$ studied in \cite{HT}.


\subsection{Global and infinitesimal compatibility}

We will deal here with a natural extension of the compatibility in Definition~\ref{def:romega} to differential forms of arbitrary degrees.
Given a differential form $\omega \in \Omega^p(M)$ and $r \in \Omega^1(M,TM)$, we denote by $\omega_r \in \Gamma(T^*M \otimes \wedge^{p-1} T^*M)$ the tensor field given by
$$
\omega_r(X_1; X_2, \dots, X_p) = \omega(r(X_1), X_2, \dots, X_p).
$$

\begin{defn}\label{dfn:nij_form}\em
We say that the pair $(\omega, r)$ is {\em compatible} if
\begin{itemize}
    \item [$(a)$] $\omega_r$ is skew-symmetric, i.e. $\omega_r \in \Omega^p(M)$,
    \item [$(b)$] $d(\omega_r) = (d \omega)_r$.
\end{itemize}

\end{defn}

A motivating example for this notion of compatibility is given in Proposition~\ref{prop:hol_dif} below. 

Denote by $\Omega_r^p(M)$ the space of differential $p$-forms $\omega$ such  that $\omega_r\in \Omega^p(M)$. Then the operator $D^{r,*}$ \eqref{eq:D*} can be extended to  
$D^{r,*}: \mathfrak{X}(M) \times \Omega_r^p(M) \to \Gamma(T^*M \otimes \wedge^{p-1} T^*M)$ 
via
\begin{equation}\label{eq:Dr_formula}
 D^{r,*}_Y(\omega) = i_Y d(\omega_r) -  i_{r(Y)} d\omega,\qquad \forall \, Y\in \mathfrak{X}(M). 
\end{equation}
So a pair $(\omega,r)$ is compatible if and only if $\omega \in \Omega_r^p(M)$ and $D^{r,*}(\omega)=0$.


\begin{defn}\label{def:infcomp}{\em
We say that an IM 2-form $(\mu,\nu)$ and an IM $(1,1)$-tensor $(D, l, r)$ on a Lie algebroid $A$ are {\em compatible} if 
\begin{align}   
\label{eq:IM1}   & \mu(l(a))  =\mu(a)_r, \,\,\qquad\quad\;\; \; \nu(l(a))  = \nu(a)_r\\
\label{eq:IM2}   & \mu(D_X(a))  = D_X^{r,*}(\mu(a)), \,\,\, \; \nu(D_X(a))  = D_X^{r,*}(\nu(a)),
\end{align}
for all $X \in \frakx(M)$ and $a \in \Gamma(A)$.}
\end{defn}

\begin{thm}\label{thm:compatibility_G}
Let $\G \toto M$ be a Lie groupoid endowed with  multiplicative $K \in \Omega^1(\G, T\G)$ and $\omega \in \Omega^p(\G)$, with corresponding infinitesimal data  $(D,l,r)$ and $(\mu, \nu)$. If the pair $(\omega, K)$ is compatible then $(D,l,r)$ and $(\mu, \nu)$ are compatible,
and the converse holds if $\G$ is source-connected.
Moreover, the IM $p$-form corresponding to $\omega_K$ is $(\mu \circ l, \nu \circ l)$.
\end{thm}

\begin{proof}
Let $\tau_1, \tau_2 \in \Gamma(\otimes^pT^*\mathcal{G})$ be tensor fields on $\mathcal{G}$ defined by
$$
\tau_1(U_1,\dots,U_p)=\omega(K(U_1),U_2,\dots, U_p), \,\,\, \tau_2(U_1,\dots, U_p) = \omega(U_1, K(U_2), U_3, \dots, U_p), $$
for $U_1, \dots, U_p \in \frakx(\G)$. Note that $\omega \in \Omega_K^p(\G)$ if and only if $\tau_1=\tau_2$. Since $\tau_1, \tau_2$ are multiplicative tensor fields, assuming that $\G$ is source connected they will be equal if and only if (see \cite[Rem.~3.16]{HT})
\begin{align}
\tag{A}  \tau_1(\cdot, \overbrace{\overrightarrow{a}}^{i\text{th-entry}}, \cdot)  & = \tau_2(\cdot, \overbrace{\overrightarrow{a}}^{i\text{th-entry}},\cdot), \,\,\quad \forall \, i=1,\dots, p,\\
\tag{B}  \mathcal{L}_{\overrightarrow{a}} \tau_1  & =\mathcal{L}_{\overrightarrow{a}} \tau_2.
\end{align} 
We first claim that (A) is equivalent to the first equation in \eqref{eq:IM1}. Indeed,  it follows from \eqref{eq:im_form} and \eqref{eq:inf_(1,1)} that
\begin{align*}
\tau_1(\overrightarrow{a},U_1,...,U_{p-1}) & = \omega(K(\overrightarrow{a}),U_1, \dots, U_{p-1}) = \omega(\overrightarrow{l(a)},U_1, \dots, U_{p-1}) = \mu(l(a))(X_1, \dots, X_{p-1}),
\end{align*}
where $X_j= Tt(U_j)$, $j=1,\dots, p$.
Similarly,
\begin{align*}
\tau_2(\overrightarrow{a},U_1,...,U_{p-1}) 
 = \mu(a)(r(X_1), X_2, \dots, X_{p-1}).
\end{align*}
Therefore $\mu(l(a))  =\mu(a)_r$ is equivalent to (A) for $i=1$.
Other values of $i$ work similarly.

Next, using that $\Lie_{\overrightarrow{a}}\omega = t^*(\nu(a) + d\mu(a))$
and \eqref{eq:inf_(1,1)}, we see that
\begin{align*}
    (\mathcal{L}_{\overrightarrow{a}} \tau_1)(U_1,U_2,...,U_p) = (\nu(a)+d\mu(a))_r(X_1,\dots, X_p) + \mu(D^r_{X_1}(a))|_{(X_2, \dots,X_p)}.
\end{align*}
Similarly,
\begin{align*}
(\mathcal{L}_{\overrightarrow{a}} \tau_2)(U_1,U_2,...U_p)&=-(\nu(a) + d\mu(a))_r(X_2,X_1,X_3,\dots, X_p) - \mu(D^r_{X_2}(a))|_{(X_1,X_3,...,X_p)}.
\end{align*}
So, if we define $\zeta: \Gamma(A) \to \Gamma(T^*M \otimes \wedge^{p-1} T^*M)$ as 
$$
\zeta(a)(Y_1; Y_2, \dots, Y_p) = (\nu(a) + d\mu(a))_r(Y_1, \dots, Y_p) + \mu(D^r_{Y_1}(a))|_{(Y_2, \dots, Y_p)}, 
$$
we see that (B) holds if and only if $\zeta$ takes values in $\Omega^p(M)$ (i.e. $\zeta(a)$ is skew-symmetric) for all $a \in \Gamma(A)$. 

When (A) and (B) hold, we have that $\tau_1=\tau_2 =\omega_K \in \Omega^{p}(\G)$ is a multiplicative $p$-form, and the corresponding IM $p$-form $(\mu_K, \nu_K)$ on $A$ is given by
$$
\mu_K(a)  = \mu(l(a)), \qquad
\nu_K(a)  = \zeta(a) - d(\mu(l(a))), 
$$
where for the second equality we used that $i_{\overrightarrow{a}}d\omega_K = \mathcal{L}_{\overrightarrow{a}} \omega_K - d i_{\overrightarrow{a}}\omega_K$. Using that $\mu(l(a)) = \mu(a)_r$ and \eqref{eq:Dr_formula},  the second equality can be rewritten as
$$
i_X \nu_K(a) =  i_X \nu(a)_r + \mu(D^r_X(a)) - D^{r,*}_X(\mu(a)).
$$

Arguing similarly with $d\omega$ in place of $\omega$, we  conclude that $d\omega \in \Omega_K^{p+1}(\G)$ 
if and only if 
$$
\nu(l(a))= \nu(a)_r
$$ 
and $\zeta^d: \Gamma(A) \to \Gamma(T^*M \otimes \wedge^{p} T^*M)$ defined by
$$
\zeta^d(a)(Y_1; Y_2, \dots, Y_{p+1}) = (d\nu(a))_r(Y_1, Y_2, \dots, Y_{p+1}) + \nu(D^r_{Y_1}(a))|_{(Y_2, \dots, Y_{p+1})} 
$$
is skew-symmetric. In this case, the IM $(p+1)$-form $(\mu_K^d, \nu_K^d)$ corresponding to $(d\omega)_K$ satisfies
$$
\mu_K^d(a) = \nu(l(a)), \,\,\quad  i_X\nu_K^d(a) = \nu(D^r_X(a)) - D^{r,*}_X(\nu(a)).
$$
When $\G$ is source-connected, we conclude that
$(d\omega)_K  = d(\omega_K)$ if and only if
$\mu^d_K(a) = \nu_K(a)$ and $\nu_K^d = 0$, which in turn is equivalent to 
$$
\mu(D^r(a))=D^{r,*}(\mu(a)),\qquad \nu(D^r(a)) = D^{r,*}(\nu(a)).
$$

\end{proof}

\subsection{Holomorphic multiplicative forms}
\label{sec:holomorphic}


In this section we specialize the results from the previous section to the holomorphic setting.



We start with a general observation relating the compatibility in Definition~\ref{dfn:nij_form} with holomorphic forms.
Let $r: TM \to TM$ be a complex structure on $M$.

\begin{prop}\label{prop:hol_dif}
 A complex form $\Omega = \omega + \ii \omega_1 \in \Omega^p(M, \C)$ is a holomorphic $p$-form  if and only if the pair $(\omega, r)$ is compatible and $\omega_1 = -\omega_r$. 
\end{prop}

\begin{proof}
The condition $\Omega \in \Omega^{(p,0)}(M)$ is equivalent to $i_{X+\ii r(X)}\Omega=0$ for all $X\in \mathfrak{X}(M)$, which in turn  holds  if and only if $\omega_1 = -\omega_r$.
We now verify that condition $(b)$ in Definition~\ref{dfn:nij_form} is equivalent to $\bar{\partial}\Omega=0$. For $Y\in \mathfrak{X}(M)$, using that $d=\partial + \bar{\partial}$ and Cartan's formula, we see that
$$
  i_{Y+\ii r(Y)}\bar{\partial}\Omega = \Lie_{Y+\ii r(Y)}\Omega =
  \Lie_{Y+\ii r(Y)}(\omega - \ii \omega_r) = A + \ii B,
$$
where $A=\Lie_Y\omega+\Lie_{r(Y)}(\omega_r)$ and $B=\Lie_{r(Y)}\omega-\Lie_Y(\omega_r)$.  Again by Cartan's formula,
we have
$$
A=i_{r(Y)}(d (\omega_r) -(d\omega)_r),\qquad B= i_Y((d\omega)_r - d (\omega_r))  
$$
So the result follows.
\end{proof}

\subsubsection{Infinitesimal description of holomorphic multiplicative forms.}
A {\em holomorphic Lie groupoid} is a Lie groupoid $\mathcal{G}\rightrightarrows M$, where both $\G$ and $M$ are complex manifolds and all the structures maps ${m}, {\epsilon}, {\textit{i}}, {s}$ and ${t}$ are holomorphic.
A {\em holomorphic Lie algebroid} is a holomorphic vector bundle $\mathcal{A} \to M$ endowed with a holomorphic bundle map $\widehat{\rho}: \mathcal{A} \to T^{1,0}M$ and a complex Lie algebra structure on the sheaf of holomorphic sections $\Gamma_{\mathrm{hol}}(\cdot, \mathcal{A})$ such that $\widehat{\rho}$ induces a morphism of sheaves of complex Lie algebras from $\Gamma_{\mathrm{hol}}(\cdot, \mathcal{A})$ to $\Gamma_{\mathrm{hol}}(\cdot, T^{1,0}M)$ and  
$$
[\sigma_1, f \sigma_2] = f [\sigma_1, \sigma_2] + (\Lie_{\widehat{\rho}(\sigma_1)}f) \sigma_2, 
$$
for $\sigma_1$, $\sigma_2$ (local) holomorphic sections of $\mathcal{A}$ and $f$ a (local) holomorphic function.

A holomorphic Lie groupoid $\G \toto M$ gives rise to a holomorphic Lie algebroid just as in the smooth real case. Since $s$ is holomorphic, $\mathcal{A}_\G:= \ker(Ts: T^{1,0}\G|_M \to T^{1,0}M)$ is a holomorphic vector bundle over $M$. It becomes a holomorphic Lie algebroid with Lie bracket $[\cdot, \,\cdot]$ induced by the one on (local) right-invariant holomorphic vector fields and $\widehat{\rho} = Tt|_{\mathcal{A}_{\G}}$.
A holomorphic Lie algebroid $\A \to M$ is said to be \textit{integrable} if $\A = \mathcal{A}_\G$ for a holomorphic Lie groupoid $\G \toto M$.

The theory of multiplicative forms and IM-forms can be naturally extended to the holomorphic context. 
A holomorphic $p$-form $\Omega$ is {\em multiplicative} if 
$$
m^*\Omega = \pr_1^*\Omega + \pr_2^*\Omega.
$$ 
Similarly, a {\em holomorphic IM $p$-form} on a holomorphic Lie algebroid $\A$ is a pair $(\widehat{\mu},\widehat{\nu})$, where $\widehat{\mu}:\A \rightarrow \wedge^{(p-1,0)}T^*M$ and $\widehat{\nu}:\A \rightarrow \wedge^{(p,0)}T^*M$ are holomorphic vector-bundle maps satisfying the IM equations
\begin{equation} \label{eq:hol_IM_eqs}
\left\{
\begin{array}{rl}
    \widehat{\mu}([u,v]) \!\! & =\mathcal{L}_{\widehat{\rho}(u)}\widehat{\mu}(v) - i_{\widehat{\rho}(v)}(\partial \widehat{\mu}(u) + \widehat{\nu}(u)),\\
    \widehat{\nu}([u,v]) \!\!& = \mathcal{L}_{\widehat{\rho}(u)}\widehat{\nu}(v) - i_{\widehat{\rho}(v)}\partial \widehat{\nu}(u),\\
    i_{\widehat{\rho}(u)}\widehat{\mu}(v) \!\!& = -i_{\widehat{\rho}(v)}\widehat{\mu}(u),
\end{array}
\right.
\end{equation}
for all $u, v$ (local) holomorphic sections of $\mathcal{A}$.

The following result is a  holomorphic version of  \cite[Thm.~4.6]{HA} describing the infinitesimal counterparts of multiplicative forms. 

\begin{thm}\label{thm:hol_dif}
Let $\mathcal{G}\rightrightarrows M$ be a holomorphic source 1-connected Lie groupoid and $\A$ its holomorphic Lie algebroid. There is a one-to-one correspondence between multiplicative holomorphic $p$-forms $\Omega$ on $\G$ and holomorphic IM $p$-forms $(\widehat{\mu}, \widehat{\nu})$ on $\A$ given by
\begin{equation}\label{eq:hol_IM_rel}
    i_{\overrightarrow{u}}\Omega = t^*\widehat{\mu}(u), \,\,\qquad  i_{\overrightarrow{u}}\partial \Omega = t^*\widehat{\nu}(u),
\end{equation}
where $\overrightarrow{u}$ is the (local) right-invariant vector field on $\G$ corresponding to a (local) holomorphic section $u$ of $\A$. 
\end{thm}

As in the real smooth case, we say that a holomorphic IM $p$-form is {\em closed} when $\widehat{\nu}$ is zero, since these are in correspondence with closed multiplicative holomorphic forms in  Theorem~\ref{thm:hol_dif}.

We will prove this theorem in $\S$ \ref{subsubsec:proof} below using Theorem~\ref{thm:compatibility_G} and the characterization of  holomorphic $p$-forms and IM $p$-forms in terms of their real parts that we discuss next (Prop.~\ref{prop:hol_IM_equations}).

\subsubsection{The underlying real geometry}\label{subsec:real}

We now expressing various holomorphic objects in terms of their underlying real structures. 

As shown in \cite[Prop.~3.3]{LMX2}, a holomorphic Lie algebroid $\mathcal{A}\to M$ can always be viewed as a real Lie algebroid  $A\to M$ together with a holomorphic structure on its underlying vector bundle such that the Lie bracket on $\Gamma(A)$ restricts to a $\C$-linear Lie bracket on the sheaf of holomorphic sections. Equivalently, if one expresses the holomorphic structure on $A$ as a Dolbeault 1-derivation $(D^{\mathrm{hol}}, l , r)$ (see Example \ref{ex:hol_str}), it is shown
in \cite[Sec.~6]{HT} that its compatibility with the Lie-algebroid structure amounts to $(D^{\mathrm{hol}}, l , r)$ being an IM $(1,1)$-tensor on $A$ (i.e., the equations in \eqref{eq:IM_1_1} hold).

A holomorphic Lie groupoid, in turn, is equivalent to a Lie groupoid $\G\toto M$ equipped with a complex structure $J$ that is multiplicative
 \cite[Prop.~3.16]{SXC}.
 For a holomorphic Lie groupoid $(\G, J)$, its real Lie algebroid $A_\G$ inherits  an IM $(1,1)$-tensor $(D,l, r)$ defined by $J$ as well as a holomorphic structure $(D^{\mathrm{hol}}, l, r)$ induced by the the isomorphism of real vector bundles 
$$
\Psi: A_\G \to \mathcal{A}_\G,\qquad \Psi(a) = \frac{1}{2}(a-\ii l(a)),
$$ 
coming from $T\G \ni U \mapsto \frac{1}{2}(U-\ii J(U)) \in T^{1,0}\G$. Notice that
\begin{equation}
\label{eq:hol_anchor}
\widehat{\rho}(\Psi(a)) = \frac{1}{2}(\rho(a) - \ii r(\rho(a))).
\end{equation}

\begin{lem}\label{lem:hol_str}
We have that $D=D^{\mathrm{hol}}$, so $(D,l,r)$ coincides with the Dolbeault 1-derivation on $A_\G$ induced by $\Psi$.
\end{lem}

\begin{proof} 
Note that $\Psi: A_\G \to \mathcal{A}_\G$ is the restriction of the usual identification $T\G \to T^{1,0}\G$, which induces the holomorphic structure on $T\G$ given by the Dolbeault 1-derivation $(D^J, J, J)$, see Example~\ref{ex:holtancot}. Hence, for $a\in \Gamma(A)$, we have that
$$
D_X^{\mathrm{hol}}(a) = D^{J}_U(\overrightarrow{a})|_M = (\Lie_{\overrightarrow{a}} J)(U)|_M = D_X(a),
$$
where $U \in \frakx(\G)$ is any vector field such that $U|_M = X$, and we used \eqref{eq:Dr} and \eqref{eq:inf_(1,1)} in the last two equalities.
\end{proof}

\begin{rmk}[Integration of holomorphic Lie algebroids]\label{rem:holinteg}\em
It is shown in \cite{LMX2} that a holomorphic Lie algebroid is integrable if and only if so is its underlying real Lie algebroid. This can be readily seen as follows. Consider a holomorphic Lie algebroid $\mathcal{A}$, given by a real Lie algebroid  $A\to M$ with Dolbeault IM $(1,1)$-tensor $(D^{\mathrm{hol}}, l, r)$.
If $\G$ is a source 1-connected Lie groupoid integrating  $A$, then $(D^{\mathrm{hol}}, l, r)$ integrates to a multiplicative $(1,1)$-tensor field $J$ on $\G$ that is a complex structure, making $\G$ into a holomorphic Lie groupoid integrating $\mathcal{A}$ (by Lemma~\ref{lem:hol_str}). 
\hfill $\diamond$
\end{rmk}

Let $\A$ be a holomorphic Lie algebroid defined by a real Lie algebroid $A \to M$ with holomorphic structure $(D^{\mathrm{hol}}, l, r)$.

\begin{prop}\label{prop:hol_IM_equations}
There is a one-to-one correspondence between  IM $p$-forms $(\mu, \nu)$ on $A$ compatible with $(D^{\mathrm{hol}}, l, r)$ (as in Def.~\ref{def:infcomp}) and holomorphic IM $p$-forms $(\widehat{\mu},\widehat{\nu})$ on $\A$ given by
\begin{equation}\label{eq:IM_real_part}
\widehat{\mu} = \mu - \ii \mu \circ l, \qquad \widehat{\nu}  = \nu - \ii \nu \circ l.
\end{equation}
\end{prop}

\begin{proof}
We start with some general observations. Let $\widehat{\gamma}: \A \to \wedge^{(p,0)} T^*M$ be a $\C$-linear vector-bundle map, and recall that the complex vector bundle underlying $\A$ is $(A,l)$. 
Writing $\widehat{\gamma} = \gamma + \ii \gamma_1$, for $\gamma,\,\gamma_1: A \to \wedge^p T^*M$, the $\C$-linearity of $\widehat{\gamma}$ is equivalent to $\gamma_1 = -\gamma \circ l$, while the fact that $\widehat{\gamma}$ takes values in $\wedge^{(p,0)} T^*M$ is equivalent to $\gamma(l(a)) = (\gamma(a))_r$. We also note that $\widehat{\gamma}$ is holomorphic
if and only if 
$$
\gamma(D^{\mathrm{hol}}_X(a)) = D^{r,*}_X(\gamma(a)).
$$
Indeed, $\widehat{\gamma}$ is holomorphic if and only if it takes holomorphic sections to holomorphic sections, which by Proposition~\ref{prop:hol_dif} is equivalent to 
$D^{r,*}_X(\gamma(a)) = 0$ for all $ X \in \frakx(M)$ and $a$ (local) holomorphic section of $A$. Since a (local) section $a$ is holomorphic if and only if $D^{\mathrm{hol}}(a)=0$, we conclude that $\gamma$ is holomorphic if and only if
\begin{equation}\label{eq:gammaD}
\gamma(D^{\mathrm{hol}}_X(a)) - D^{r,*}_X(\gamma(a)) = 0
\end{equation}
for all $X \in \frakx(M)$ and $a$ (local) holomorphic section.
Now note that the map $a\mapsto \gamma(D^{\mathrm{hol}}_X(a)) - D^{r,*}_X(\gamma(a))$ is $C^\infty(M)$-linear, so \eqref{eq:gammaD}
holds for all smooth sections $a$ since $\Gamma(A)$ is locally generated by holomorphic sections as a $C^\infty(M)$-module.

It immediately follows  that any holomorphic IM $p$-form $(\widehat{\mu}, \widehat{\nu})$ is determined via \eqref{eq:IM_real_part} by a pair of vector bundle maps $\mu:A \to \wedge^{p-1} T^*M$, $\nu: A \to \wedge^p T^*M$ satisfying \eqref{eq:IM1} and \eqref{eq:IM2}. So it remains to prove that the pair $(\widehat{\mu}, \widehat{\nu})$ satisfies the holomorphic IM-equations \eqref{eq:hol_IM_eqs} if and only if $(\mu, \nu)$ satisfies the IM-equations \eqref{eqmubr}. For $a$ and $b$ (local) holomorphic sections of $A$, define
$$
\Upsilon(a,b) := \widehat{\mu}([a, b]) - \Lie_{\widehat{\rho}(a)} \widehat{\mu}(b) + i_{\widehat{\rho}(b)} \left(\partial \widehat{\mu}(a)+\widehat{\nu}(a)\right). 
$$
From \eqref{eq:hol_anchor} and \eqref{eq:IM_real_part}, it follows that $\Upsilon(a,b) = \Upsilon_0(a,b) - \ii \Upsilon_1(a,b)$, where
\begin{align*}
\Upsilon_0(a,b)  = &  \mu([a,b]) -  \frac{1}{2}\left(\Lie_{\rho(a)}\mu(b) - i_{\rho(b)}(d\mu(a)+ \nu(b))\,\right)\\
&  +\frac{1}{2}\left(\Lie_{r(\rho(a))}\mu(l(b)) + i_{r(\rho(b))}(d\mu(l(a)) + \nu(l(a))) \,\right),\\
\Upsilon_1(a,b)  = &  \mu(l([a,b]) + \frac{1}{2}\left(\Lie_{r(\rho(a))}\mu(b) - i_{r(\rho(b))}(d\mu(a)+ \nu(a)) \,\right)\\
&   + \frac{1}{2}\left(\Lie_{\rho(a)}\mu(l(b)) - i_{\rho(b)} (d\mu(l(a))+ \nu(l(a)))\,\right). 
\end{align*}
Since $D^{r,*}(\mu(a)) = \mu(D^{\mathrm{hol}}(a)) = 0$ for $a$ holomorphic, we have that
$d(\mu(l(a))) = d(\mu(a)_r) = (d\mu(a))_r$, which implies that
\begin{align*}
&i_{r(\rho(b))} d\mu(l(a)) = - i_{\rho(b)}d\mu(a),\\
&\Lie_{r(\rho(a))} \mu(l(b)) = - \Lie_{\rho(a)}\mu(b).
\end{align*}
Hence, since $\nu(l(a)) = \nu(a)_r$, we obtain that
$$
\Upsilon_0(a,b) = \mu([a,b]) - \Lie_{\rho(a)}\mu(b) + i_{\rho(b)}(d\mu(a)+\nu(a)).
$$
By a similar argument, and using the fact that $
l([a,b]) = [a,l(b)]$ for $a$, $b$ holomoprhic (as a consequence of \eqref{eq:IM_1_1}), one has that
$$
\Upsilon_1(a,b) =  \mu([a,l(b)]) - \Lie_{\rho(a)}\mu(l(b)) + i_{\rho(l(b))}(d\mu(a) + \nu(a)).
$$
It follows that $\Upsilon(a,b) = 0$ for all $a$, $b$ (local) holomorphic sections if and only if
$$ 
\mu([a,b]) = \Lie_{\rho(a)}\mu(b) - i_{\rho(b)}(d\mu(a) + \nu(a)),\,\,\;\; \forall \, a,b \in \Gamma(A),
$$
since $\Upsilon$ is $C^{\infty}(M)$-bilinear as a function of $a, \, b$ and $\Gamma(A)$ is locally generated by holomorphic sections as a $C^\infty(M)$-module. The analogous result for $\widehat{\nu}$ is proven similarly.
\end{proof}

\subsubsection{Proof of Theorem~\ref{thm:hol_dif}}\label{subsubsec:proof}

Write a holomorphic  $p$-form $\Omega$ on $(\G, J)$ as $\Omega = \omega - \ii \omega_J$, with $(\omega, J)$ a compatible pair, see Proposition \ref{prop:hol_dif}. Notice that $\Omega$ is multiplicative if and only if $\omega$ is multiplicative, and assume that this holds. Let 
$(\mu,\nu)$ be the IM-form corresponding to $\omega$ via \eqref{eq:im_form}. For a local holomorphic section
$a$ of $A_\G$, we have
\begin{align}
\label{eq:hol_IM_rel_1}i_{\overrightarrow{\Psi(a)}}\Omega & = i_{\overrightarrow{a}}\omega- \ii i_{\overrightarrow{l(a)}}\omega = t^*(\mu(a) - \ii \mu(l(a))),\\
\label{eq:hol_IM_rel_2}i_{\overrightarrow{\Psi(a)}}\partial\Omega & =  i_{\overrightarrow{a}}d\omega -\ii i_{\overrightarrow{a}} d(\omega_J) = t^*(\nu(a) - \ii \nu(l(a))),
\end{align}
where we have used that $J(\overrightarrow{a})=\overrightarrow{l(a)}$ and the compatibility $d(\omega_J) = (d\omega)_J$. Define $(\widehat{\mu},\widehat{\nu})$ by \eqref{eq:IM_real_part}. It is clear that \eqref{eq:hol_IM_rel} holds for $(\widehat{\mu},\widehat{\nu})$ and it is a holomorphic IM $p$-form as consequence of the compatibility of $(\omega, J)$ via Theorem \ref{thm:compatibility_G}  and Proposition \ref{prop:hol_IM_equations}. Conversely, given a holomorphic IM $p$-form $(\widehat{\mu},\widehat{\nu})$, write it as in \eqref{eq:IM_real_part} for a real  IM $p$-form $(\mu,\nu)$. When $\G$ is source 1-connected, there is a real multiplicative $p$-form $\omega$ integrating $(\mu, \nu)$. From Theorem \ref{thm:compatibility_G} and Lemma~\ref{lem:hol_str}, one has that $(\omega, J)$ is compatible, so
 $\Omega = \omega - \ii \omega_J$ is a holomorphic $p$-form by Proposition~\ref{prop:hol_dif}. Moreover, it follows from \eqref{eq:hol_IM_rel_1} and \eqref{eq:hol_IM_rel_2}
that \eqref{eq:hol_IM_rel} holds.

\begin{rmk}\label{rem:realparts}\em
Note that a central fact shown in the proof of Theorem \ref{thm:hol_dif} is that a multiplicative holomorphic $p$-form $\Omega \in \Omega^{(p,0)}(\G)$ integrates a holomorphic IM $p$-form $(\widehat{\mu}, \widehat{\nu})$  if and only if $\omega$, the real part of $\Omega$, integrates $(\mu, \nu)$, the real part of $(\widehat{\mu}, \widehat{\nu})$. 
\hfill $\diamond$
\end{rmk}


\section{Integration of Dirac-Nijenhuis structures}
\label{sec:int_dirac}
In this section we refine the inifinitesimal-global correspondence between Dirac structures and presymplectic groupoids to include compatible with $(1,1)$-tensor fields. In this way we obtain the global counterparts of Dirac-Nijenhuis structures and, as a particular case, the integration of holomorphic Dirac structures (see  \cite{SXC,SX} for the special cases of Poisson (quasi-) Nijenhuis and holomorphic Poisson structures).

\subsection{Dirac structures and presymplectic groupoids}\label{subsec:presymp}

We now briefly recall the integration of Dirac structures to presymplectic groupoids \cite{bcwz}, as an application of the correspondence between multiplicative and IM-forms in \cite{HA}. 

The starting point is the observation that Dirac structures can be identified with special types of IM 2-forms. 
Any Dirac structure $L \subset \T M$ is naturally a Lie algebroid with respect to the restriction of the Courant bracket $\Cour{\cdot,\cdot}$ to $\Gamma(L)$ and anchor $\rho = \pr_T|_L: L \to TM$. One can then directly verify that the  projection $\pr_{T^*}|_L: L \to T^*M$  defines a closed IM 2-form on $L$. Conversely, the closed IM 2-forms that arise in this way have a simple characterization.
Let $\mu: A \to T^*M$ be a closed IM 2-form on a Lie algebroid $A$ such that
\begin{equation}\label{eq:transv_condition}
\begin{cases}
\mathrm{rank}(A)= \dim(M),\\ 
\ker(\rho) \cap \ker(\mu) = \{0\}.
\end{cases}
\end{equation}
A consequence of the IM-equations \eqref{eqmubr} (with $\nu=0$) is that the image of the map 
$(\rho, \mu): A \to \T M$, 
\begin{equation}\label{eq:inducedL}
L:=\mathrm{Im}(\rho,\mu),
\end{equation}
is a Dirac structure. The Lie algebroid $A$ is identified with $L$ under this map, in such a way that $\mu=\mathrm{pr}_{T^*}|_L$.

A {\em presymplectic groupoid} is a Lie groupoid $\G\toto M$ equipped with a closed multiplicative 2-form $\omega$ such that 
\begin{equation}\label{Kercond}
\begin{cases}
 \dim(\G) = 2 \dim(M)\\
    \ker(Ts)_x\cap \ker(Tt)_x\cap \ker(\omega)_x=\{0\}, \,\,\, \forall \, x\in M. 
\end{cases}
\end{equation}

It is shown in \cite[Sec.~4]{HA} that conditions \eqref{Kercond} on $\G$ and $\omega$ are equivalent to conditions
\eqref{eq:transv_condition} for the corresponding Lie algebroid $A$ and  closed IM 2-form $\mu: A\to T^*M$. 
It follows that if $(\G,\omega)$ is a presymplectic groupoid, then $M$ inherits a Dirac structure $L$, as in \eqref{eq:inducedL}, together with an induced isomorphism of Lie algebroids between $L$ and $A$. Moreover, the first equation in \eqref{eq:im_form} implies that $L$ is uniquely determined by the fact that $t: \G\to M$ is a forward Dirac map. When this happens, we say that the presymplectic groupoid $(\G,\omega)$ {\em integrates} the Dirac structure $L$.

The opposite direction, going from Dirac structures to presymplectic groupoids,  follows from the integration of IM-forms (see $\S$ \ref{subsec:multiplicative}): given a Dirac structure on $L$ that is integrable as a Lie algebroid and its source 1-connected integration $\G$, 
the integration of the IM 2-form $\mu=\pr_{T^*}|_L$ defines a
unique (up to isomorphism) closed multiplicative 2-form $\omega$ such that $(\G,\omega)$ is a presymplectic groupoid integrating $L$.


\subsection{Integration of Dirac structures with compatible $(1,1)$-tensor fields}

We will now relate compatible $(1,1)$-tensor fields on presymplectic groupoids and on their associated Dirac manifolds.

\begin{lem}\label{lem:rIM}
Let $L$ be a Dirac structure on $M$, $r\in \Omega^1(M,TM)$,
so that the pair $(L,r)$ is compatible. Then the triple $(\D^r|_{\Gamma(L)}, (r,r^*)|_L, r)$ is an IM $(1,1)$-tensor on the Lie algebroid  $L$. Moreover, if $r$ is Nijenhuis, then this IM $(1,1)$-tensor satisfies the infinitesimal Nijenhuis equations \eqref{eq:im_nijenhuis}.
\end{lem}

\begin{proof}
The Leibnz-type identity for $(\D^r|_{\Gamma(L)}, (r,r^*)|_L, r)$ follows from \eqref{eq:leibniz_D}. The IM-equations  will be a consequence of the following general equations relating $(\D^r, (r,r^*), r)$ with the Courant-algebroid structure on $\T M$:
\begin{align}
\label{eq:Cour_im_1} \pr_T \circ (r,r^*)   = & r \circ \pr_T\\
\label{eq:Cour_im_2} \pr_T(\D^r_X(\sigma_1))  = & D^{r}_X(\pr_T(\sigma_1))\\ 
\label{eq:Cour_im_3}(r,r^*)(\Cour{\sigma_1,\sigma_2})  = & \Cour{\sigma_1, (r,r^*)(\sigma_2)} - \D^r_{\pr_T(\sigma_2)}(\sigma_1) - C^r(\sigma_1,\sigma_2)\\
\label{eq:Cour_im_4} \D^r_Z(\Cour{\sigma_1, \sigma_2}) = & \Cour{\sigma_1, \D^r_Z(\sigma_2)} - \Cour{\sigma_2, \D^r_Z(\sigma_1)} + \D^r_{[\pr_T(\sigma_2), Z]}(\sigma_1) \\
\nonumber &  - \D^r_{[\pr_T(\sigma_1),Z]}(\sigma_2) - i_Z dC^r(\sigma_1,\sigma_2)
\end{align}
where $Z \in \frakx(M)$, $\sigma_1, \, \sigma_2 \in \Gamma(\T M)$, and $C^r(\sigma_1, \sigma_2) = \<\D^r_{(\cdot)}(\sigma_1), \sigma_2\> \in \Omega^1(M)$. To obtain the IM equations \eqref{eq:IM_1_1}, just observe that when $(L,r)$ is a Dirac structure, $C^r(\sigma_1, \sigma_2)=0$ for $\sigma_1, \sigma_2 \in \Gamma(L)$.

The first two equations above are straightforward. The other two can be verified as follows. Let $\sigma_1 = (X, \alpha)$, $\sigma_2 = (Y, \beta)$. Note that
 $$
 \<\beta, D_{(\cdot)}^r(X)\> = \Lie_{X} (r^*(\beta)) - r^*(\Lie_X\beta), \qquad \<D^{r,*}_{(\cdot)}(\alpha), Y\> = r^*(i_Yd\alpha)-i_Ydr^*(\alpha).
 $$
Then
 \begin{align*}
 (r,r^*)(\Cour{\sigma_1,\sigma_2}) - \Cour{\sigma_1, (r,r^*)(\sigma_2)} & = -(D^r_Y(X),\,\<\beta, D^r_{(\cdot)}(X)\> + r^*(i_Y d\alpha) - i_{r(Y)}d\alpha)\\
 & = -\D^r_Y(X,\alpha) - \<\beta, D^r_{(\cdot)}(X)\> - \<D^{r,*}_{(\cdot)}(\alpha), Y\>,
 \end{align*}
 which proves \eqref{eq:Cour_im_3}. To verify \eqref{eq:Cour_im_4}, define
 $$
 \Sigma(Z,\sigma_1,\sigma_2) = \mathbb{D}^r_Z(\Cour{ \sigma_1,\sigma_2})  - \Cour{\sigma_1,\mathbb{D}^r_Z(\sigma_2)} + \Cour{\sigma_2,\mathbb{D}^r_Z(\sigma_1)} - \mathbb{D}^r_{[\pr_T(\sigma_2),Z]}(\sigma_1) + \mathbb{D}^r_{[\pr_T(\sigma_1),Z]}(\sigma_2).
 $$
It follows from the Jacobi identity for the Lie bracket of vector fields that the $TM$-component of $\Sigma$ vanishes. The component on $T^*M$ can written as
$$
 \pr_{T^*}(\Sigma(Z,\sigma_1,\sigma_2)) = \Upsilon(Z, X, \beta) + \Upsilon(Z, Y, \alpha),
$$
where
\begin{align*}
 \Upsilon(Z,X,\beta) = &  \mathcal{L}_Z(r^*(\mathcal{L}_X\beta))  -\mathcal{L}_{r(Z)}\mathcal{L}_X\beta 
 -\mathcal{L}_X \mathcal{L}_Z(r^*(\beta)) + \mathcal{L}_X\mathcal{L}_{r(Z)}\beta 
   -i_{[X,r(Z)]}d\beta\\
 &  + i_{r([X,Z])}d\beta  + \mathcal{L}_{[X,Z]}(r^*(\beta)) -\mathcal{L}_{r([X,Z])}\beta,\\
\Upsilon(Z,Y,\alpha)   = & - \mathcal{L}_Z(r^*(i_Yd\alpha)) + \mathcal{L}_{r(Z)}i_Yd\alpha +i_{[Y,r(Z)]}d\alpha -i_{r([Y,Z])}d\alpha 
+\mathcal{L}_Y \mathcal{L}_Zr^*(\alpha)  \\
&   -\mathcal{L}_Y\mathcal{L}_{r(Z)}\alpha - \mathcal{L}_{[Y,Z]}r^*(\alpha) +\mathcal{L}_{r([Y,Z])}\alpha.
\end{align*}
Using Cartan calculus and \eqref{eq:dualD}, one verifies that
\begin{align*}
\Upsilon(Z,X,\beta) & =  d\<\beta, D_Z^{r}(X)\> +  \Lie_Z\underbrace{(r^*(\Lie_X \beta) - \Lie_X r^*\beta)}_{-\<\beta, D_{(\cdot)}^{r}(X)\>}  = - i_Zd\<\beta, D_{(\cdot)}^r(X)\>\\
\Upsilon(Z,Y,\alpha) & = d(\underbrace{ \Lie_Z\<\alpha, r(Z)\> - \Lie_{r(Z)}\<\alpha, Y\> - \<\alpha, D^r_Z(Y)\>}_{\<D^{r,*}_Z(\alpha), Y\>}) - \Lie_Z\underbrace{(r^*(i_Yd\alpha) - i_Ydr^*(\alpha))}_{\<D^{r,*}_{(\cdot)}(\alpha),Y\>}\\
 & = - i_Zd\<D_{(\cdot)}^{r,*}(\alpha), Y\>.
\end{align*}
This proves the first part of the lemma.

For the second assertion regarding the infinitesimal Nijenhuis equations \eqref{eq:im_nijenhuis} for the IM $(1,1)$-tensor $(\D^r|_{\Gamma(L)}, (r,r^*)|_L, r)$, first note that
\begin{align*}
 (r,r^*)(\D^r_X(Z,\gamma)) - \D^r_X((r,r^*)(Z,\gamma))  =   (\N_r(X,Z), \<\gamma, \N_r(X,\cdot)\>)
\end{align*}
for all $X, Z \in \frakx(M), \, \gamma \in \Omega^1(M)$, as a consequence of \eqref{Nij_D} and \eqref{eq:dualNij}. Similarly, one can check, using \eqref{Nij_D}, \eqref{eq:dualNij} and the Jacobi identity for the Lie bracket of vector fields, that
\begin{equation}\label{eq:Dr_square}
(\D^r)^2_{(X,Y)}(Z,\gamma) =  ((\Lie_Z \N_r)(X,Y),  i_Xi_Y d\N_r^*\gamma - i_{\N_r(X,Y)}d\gamma),
\end{equation} 
where $\N_r^*: T^*M \to \wedge^*T^*M$ is the dual of the Nijenhuis torsion and $(\D^r)^2$ is defined in \eqref{eq:im_nijenhuis}. In particular, if $\N_r=0$, then equations \eqref{eq:im_nijenhuis} hold for $(\D^r|_{\Gamma(L)}, (r,r^*)|_L, r)$.

\end{proof}

\begin{rmk}[Courant-Nijenhuis algebroids]\em Equations \eqref{eq:Cour_im_1}--\eqref{eq:Cour_im_4} make sense for arbitrary Courant algebroids endowed with a self-dual 1-derivation (in place of $(\D^r, (r,r^*), r)$) and lead to  the general notion of  {\em Courant-Nijenhuis algebroid}, studied in \cite{BDC} (which arise, in particular, as ``doubles'' of the Lie-Nijenhuis bialgebroids  in \cite{T}).
\hfill $\diamond$
\end{rmk}

We are now ready to show the infinitesimal-global correspondence of presymplectic-Nijenhuis groupoids and Dirac-Nijenhuis structures.

\begin{thm}\label{thm:main_integration}
Let $(\G,\omega)$ be a presymplectic groupoid integrating a Dirac structure $L$ on $M$.
\begin{itemize}
    \item[(a)]  Let $K\in \Omega^1(\G,T\G)$ be multiplicative, and let $r=K|_{TM}: TM \to TM$.  
    If the pair $(\omega, K)$ is compatible (resp. presymplectic-Nijenhuis) then the pair $(L,r)$ is compatible (resp. Dirac-Nijenhuis). 
    \item [(b)] If $\G$ is source 1-connected, the correspondence $K \mapsto r=K|_{TM}$ is a bijection between multiplicative $K\in \Omega^1(\G,T\G)$ compatible with $\omega$ and $r\in \Omega^1(M,TM)$ compatible with $L$ (moreover, $K$ is Nijenhuis if and only if so is $r$).
\end{itemize}

\end{thm}

\begin{proof}
Suppose that the pair $(\omega,K)$ is compatible.
Let $(D, l, r)$ be the IM $(1,1)$-tensor corresponding to $K$ and $\mu$ the closed IM 2-form of $\omega$, so that $L=(\rho,\mu)(A)$. 
It follows from the first two equalities in \eqref{eq:IM_1_1},
$$
r\circ \rho = \rho\circ l, \qquad \rho(D_X(a)) = D_X^r(\rho(a)),
$$
and the compatibility conditions \eqref{eq:IM1} and \eqref{eq:IM2} (which hold by Thm.~\ref{thm:compatibility_G}, recalling that $\nu=0$),
$$
r^*\mu(a) = \mu(l(a)),\qquad \mu(D_X(a))=D_X^{r,*}(\mu(a)),
$$
 that the pair $(L,r)$ is compatible. Moreover, if $K$ is Nijenhuis then so is $r$ (see \eqref{eq:im_nijenhuis}). This proves part (a). Note also that, under the Lie-algebroid isomorphism $(\rho,\mu): A\stackrel{\sim}{\to} L$, the IM $(1,1)$-tensor $(D,l,r)$ is identified with $(\D^ r|_{\Gamma(L)}, (r,r^*), r)$. As a consequence, if $\G$ is source connected, any $K'$ that is multiplicative, compatible with $\omega$, and satisfies  $K'|_{TM}=r$, must agree with $K$.

Suppose now that $(\G,\omega)$ is a source 1-connected integration of $L$.
Let $r\in \Omega^ 1(M,TM)$ be compatible with $L$. By Lemma~ \ref{lem:rIM},
$(\D^ r|_{\Gamma(L)}, (r,r^*), r)$ is an IM $(1,1)$-tensor on $L$, which is easily seen to be compatible with the closed IM 2-form $\pr_{T^*}|_L$.
It follows that the IM $(1,1)$-tensor $(D,l,r)$ corresponding to $(\D^ r|_{\Gamma(L)}, (r,r^*), r)$ through the identification $(\rho,\mu): A\stackrel{\sim}{\to} L$ is compatible with $\mu$. Hence $K\in \Omega^ 1(\G,T\G)$ obtained by integration of $(D,l,r)$ is compatible with $\omega$
by Thm.~ \ref{thm:compatibility_G}. Such $K$ satisfies $K|_{TM}=r$ and, as argued above, any multiplicative $(1,1)$-tensor field on $\G$ compatible with $\omega$ with this property must coincide with $K$. The fact that $K$ is Nijenhuis if $r$ is follows from the second assertion in Lemma~\ref{lem:rIM}.

\end{proof}

When (a) holds, we say that $(\G, \omega, K)$ {\em integrates} $(M, L, r)$. 

\begin{rmk}[The quasi-Nijenhuis version]\label{rem:quasiint}\em
The correspondence in Theorem \ref{thm:main_integration} has the following version for quasi-Nijenhuis structures (see Remark~\ref{rem:quasi}).
$$
(L,r,\phi) \text{ is Dirac quasi-Nijenhuis } 
\Longleftrightarrow \omega(\cdot, \N_K(\cdot, \cdot)) = s^*\phi - t^*\phi.
$$
Indeed, by considering 
$
\tau(U,V,W) = \omega(U, \N_K(V,W))
$
and its infinitesimal description, one can show that $\tau$ is skew-symmetric (hence a 3-form) with corresponding IM 3-form $\tilde{\mu}: A \to \wedge^2 T^*M$, $\tilde{\nu}: A \to \wedge^3 T^*M$ on $A$ given by
\begin{align*}
\tilde{\mu}(a) & = \N_r^*(\mu(a))\\
\tilde{\nu}(a)|_{(X,Y,Z)} & = d\mu(a)(X, \N_r(Y, Z)) - \<\mu(D^2_{(Y,Z)}(a)), X\> - d(\N_r^*\mu(a))(X,Y,Z).
\end{align*}
Since the pair $(\omega,K)$ is compatible, it follows from Theorem \ref{thm:compatibility_G} and \eqref{eq:Dr_square} that
$$
\mu(D^2_{(Y,Z)}(a)) = (D^{r,*})^2_{(Y,Z)}(\mu(a)) = i_Yi_Zd(\N_r^*\mu(a)) - i_{\N_r(Y,Z)}d\mu(a),  
$$
which implies that $\tilde{\nu}=0$.
Also, the quasi-Nijenshui condition \eqref{eq:quasi} is equivalent to 
$
\tilde{\mu}(a) = - i_{\rho(a)}\phi,
$
which is the closed IM 3-form corresponding to $s^*\phi - t^*\phi$.
\end{rmk}

In conclusion, we obtain a natural bijective correspondence between source 1-connected presymplectic-Nijenhuis  groupoids and  Dirac-Nijenhuis manifolds (which are integrable as Lie algebroids), 
$$
(\G,\omega, K)\;\; \leftrightharpoons\;\; (M, L, r),
$$
where $L$ is the forward image of $\omega$ by $t$ and $r$ is $t$-related with $K$ (and the same holds for the quasi-Nijenhuis condition).
This restricts to a correspondence between symplectic (quasi-)Nijenhuis groupoids and Poisson (quasi-)Nijenhuis structures, recovering \cite[Thm.~5.2]{SX}.

We have the following simple relation between integration of Dirac-Nijenhuis structures and the ``$(0,n)$'' hierarchies described in $\S$ \ref{subsec:hierarchy}.

\begin{cor}\label{cor:(0,1)_integration}
Suppose that $(\G, \omega, K)$ is a presymplectic-Nijenhuis groupoid integrating the Dirac-Nijenhuis structure $(L, r)$. 
If $\ker(\id_{TM}, r^*)|_L)=0$, then for each $n=1,2,\ldots$  $(\G,\omega_n)$ is a presymplectic groupoid integrating the Dirac structure $L_{(0,n)}$.
\end{cor}

\begin{proof}
The proof follows from the fact that, if $\ker(\id_{TM}, r^*)|_L)=0$, then $(\id_{TM},(r^*)^n): L \to L_{(0,n)}$ is a Lie algebroid isomorphism. Hence, as $L=(\rho, \mu)(A)$, it follows that $L_{(0,n)}= (\rho, (r^*)^n\circ \mu)(A)=(\rho, \mu\circ l^n)(A)$, which is the Dirac structure corresponding to the presymplectic groupoid $(\G,\omega_n)$.
\end{proof}

\begin{ex} \em  
Let $(\G, \omega)$ be a symplectic groupoid integrating a Poisson manifold $(M,\pi)$. Consider a closed 2-form $B \in \Omega^2(M)$ and $r=\id + \pi^\sharp\circ B^\flat: TM \to TM$, as in Example \ref{ex:gauge}. One can check that the IM $(1,1)$-form $(D^{r,*}, r^*, r)$ on the cotangent Lie algebroid $(T^*M, [\cdot,\cdot]_\pi, \pi^\sharp)$ integrates to $K=\id_{T\G} +\Pi^\sharp \circ (t^*B - s^*B)^{\flat}$. As observed in Example \ref{ex:gauge}, the gauge transformation of $\pi$ by $B$ is $L_{(0,1)}$,   and hence by  Corollary \ref{cor:(0,1)_integration} it  integrates to $(\G, \omega_1 = \omega+ t^*B-s^*B)$, thus recovering  \cite[Thm.~4.1]{BurRadko}. When $r$ is invertible, $K$ is also invertible and $\omega_1$ is the symplectic form integrating the Poisson structure resulting from the gauge transformation of $\pi$ by $B$.
\end{ex}


\subsection{Integration of holomorphic Dirac structures}

We now introduce the global counterparts of holomorphic Dirac structures.

\begin{defn}\label{def:holgrp}
A \textit{holomorphic presymplectic groupoid} is a holomorphic Lie groupoid $\G \toto M$ endowed with a multiplicative holomorphic $2$-form $\Omega$ such that $\partial \Omega =0$,  $\dim(\G) = 2 \dim(M)$, and
\begin{equation}\label{eq:kernel_intersection}
\ker(\Omega)_x \cap \ker(Ts)_x \cap \ker(Tt)_x  = \{0\},\qquad \forall \, x \in M.
\end{equation}
\end{defn}

The intersection condition  \eqref{eq:kernel_intersection} takes place in $T^{1,0}\G |_M$. We denote by $\widehat{\mu}: \A \to (T^{1,0}M)^*$ the holomorphic IM 2-form corresponding to $\Omega$, $\widehat{\mu}(u) = (i_u\Omega)|_{T^{1,0}M}$ (note that $\widehat{\nu}=0$ since $\partial \Omega=0$).

We will prove the infinitesimal-global correspondence relating holomorphic Dirac structures and holomorphic presymplectic groupoids as a consequence of Theorem~\ref{thm:hol_dif}, though it can also be derived from Theorem~\ref{thm:main_integration} by working with real parts (see comments after Thm.~\ref{thm:integholDirac} below).

Just as in the real smooth case, holomorphic presymplectic groupoids give rise to holomorphic Dirac structures on their spaces of units.

\begin{thm}\label{thm:diffpresgrp}
Let $(\G, \Omega) \toto M$ be a holomorphic presymplectic groupoid. (1) There exists a unique holomorphic Dirac structure $\mathcal{L} \subset \T^{1,0} M$ for which $t: \G \to M$ is a forward Dirac map. (2) The map $(\widehat{\rho},\widehat{\mu}): \A \to  \T^{1,0}M$ induces  an isomorphism of Lie algebroids $\A \stackrel{\sim}{\to} \mathcal{L}$.
\end{thm}

\begin{proof}
The proof is analogous to the real case.
Holomorphic Dirac structures $\mathcal{L}\subset \T^{1,0}M$ can be identified with closed holomorphic IM 2-forms $\widehat{\mu}: \mathcal{A}\to (T^{1,0}M)^*$ satisfying 
\begin{equation}\label{eq:holconds}
    \mathrm{rank}(\mathcal{A})=\dim_\C(M) , \qquad \ker(\widehat{\rho})\cap \ker(\widehat{\mu})=\{0\},
\end{equation}
(cf. $\S$ \ref{subsec:presymp}) as the image of $(\widehat{\rho},\widehat{\mu}): \mathcal{A}\to \T^{1,0} M$,
\begin{equation}\label{eq:Lhol}
\mathcal{L}=\mathrm{Im}(\widehat{\rho},\widehat{\mu}).
\end{equation}
The multiplicativity of $\Omega$ implies that
\begin{equation}\label{eq:Omegaright}
i_{\overrightarrow{u}}\Omega = t^*\widehat{\mu}(u),
\end{equation}
for any section of $\Gamma(\A)$. By using this last property, one can check that 
 \eqref{eq:kernel_intersection} boils down to the intersection condition in \eqref{eq:holconds}, so one obtains a holomorphic Dirac structure $\mathcal{L}$ as in 
 \eqref{eq:Lhol}. 
   Moreover, $(\widehat{\rho},\widehat{\mu}): \A \stackrel{\sim}{\to} \mathcal{L}$ is an isomorphism of Lie algebroids.
 
Using the identity
\begin{equation}\label{eq:rhoT}
Tt(\overrightarrow{u})=\widehat{\rho}(u),
\end{equation} 
it also follows from \eqref{eq:Omegaright} that $t: (\G,\Omega)\to (M,L)$ is a forward Dirac map. 

\end{proof}

We say that a holomorphic pre-symplectic groupoid $(\G, \Omega)$ {\em integrates} a holomorphic Dirac structure $\mathcal{L}$ if $t:(\G, \Omega) \to (M,\mathcal{L})$ is a forward Dirac map.
The next result gives the integration of holomorphic Dirac structures.

\begin{thm}\label{thm:integholDirac}
Let $\mathcal{L} \subset \T^{1,0} M$ be a holomorphic Dirac structure whose underlying holomorphic Lie algebroid is integrable, and let $\G \toto M$ be its  source 1-connected holomorphic integration. Then there exists a unique holomorphic multiplicative $2$-form $\Omega$ such that $(\G, \Omega)$ is the pre-symplectic groupoid integrating $\mathcal{L}$.  
\end{thm}

\begin{proof}
Set $\widehat{\mu}:=\pr_{T^*}: \mathcal{L} \to (T^{1,0}M)^*$, the projection of $\mathcal{L}$ on the cotangent component.
The fact that $\mathcal{L}$ is a holomorphic Dirac structure implies that $\hat{\mu}$ is a holomorphic IM $2$-form on $\mathcal{L}$. By Theorem \ref{thm:hol_dif}, there exists a unique multiplicative holomorphic $2$-form $\Omega$ on $\G$ integrating $(\widehat{\mu}, \widehat{\nu}=0)$. It is clear that $\partial \Omega =0$ and $\dim(\G) = 2 \dim(M)$, and $t: (\G,\Omega)\to (M,\mathcal{L})$ is a forward Dirac map (as a consequence of \eqref{eq:Omegaright} and \eqref{eq:rhoT}). Since $\widehat{\rho}=\mathrm{pr}_T$, we have that $\ker(\widehat{\rho})\cap \ker(\widehat{\mu}) =\{0\}$, which is equivalent to  $\ker(\Omega)\cap \ker(Tt)\cap \ker(Ts)|_M=\{0\}$. 

If $\Omega$ and $\Omega'$ are two presymplectic integrations of $\mathcal{L}$, then by Theorem~\ref{thm:diffpresgrp}, part (2), there is a (holomorphic) Lie-algebroid automorphism $\phi$ such that $\widehat{\mu}'=\widehat{\mu}\circ \phi$. Since $\G$ is source 1-connected, there is an automorphism of $F:\G\to \G$ such that $F^*\Omega=\Omega'$.

\end{proof}


Theorems~\ref{thm:diffpresgrp} and \ref{thm:integholDirac} extend the  correspondence between (integrable) holomorphic Poisson structures and source 1-connected holomorphic symplectic groupoids shown in \cite{LMX2}.

Another approach to the previous two theorems relies on the underlying real structures. A holomorphic Dirac structure $\mathcal{L}$ on a complex manifold $(M,r)$ is integrable as a holomorphic Lie algebroid if and only if its real part $L$ is integrable as a real Lie algebroid (see Remark~\ref{rem:holinteg}). From the bijective correspondence between holomorphic Dirac structures $\mathcal{L}$ and Dirac-Nijenhuis structures $(L,r)$ in Proposition~\ref{prop:hol_dirac}, one obtains an equivalence between presymplectic-Nijenhuis structures  and closed holomorphic 2-forms on $M$ via $\omega \mapsto \omega - \ii \omega_r$, see Example~\ref{ex:presymphol}.
 As in Remark~\ref{rem:realparts}, one can check that if $(\G,\Omega)$ is a holomorphic presymplectic groupoid integrating $\mathcal{L}$, then $(\G, \omega, J)$ is a presymplectic-Nijenhuis groupoid integrating $(L, r)$, so passing to real parts is compatible with differentiation and integration:

\begin{figure}[h!]
\label{fig2}
\centering

\tikzset{every picture/.style={line width=0.75pt}} 

\begin{tikzpicture}[x=0.75pt,y=0.75pt,yscale=-1,xscale=1]

\draw   (122.2,105.81) -- (279.21,105.81) -- (279.21,68.4) -- (122.2,68.4) -- cycle ;
\draw    (291.2,78.4) -- (350.21,78.8) ;
\draw [shift={(352.21,78.81)}, rotate = 180.39] [color={rgb, 255:red, 0; green, 0; blue, 0 }  ][line width=0.75]    (10.93,-3.29) .. controls (6.95,-1.4) and (3.31,-0.3) .. (0,0) .. controls (3.31,0.3) and (6.95,1.4) .. (10.93,3.29)   ;
\draw    (352.2,98.4) -- (291.21,97.83) ;
\draw [shift={(289.21,97.81)}, rotate = 360.53999999999996] [color={rgb, 255:red, 0; green, 0; blue, 0 }  ][line width=0.75]    (10.93,-3.29) .. controls (6.95,-1.4) and (3.31,-0.3) .. (0,0) .. controls (3.31,0.3) and (6.95,1.4) .. (10.93,3.29)   ;
\draw   (370.2,103.81) -- (527.21,103.81) -- (527.21,66.4) -- (370.2,66.4) -- cycle ;
\draw   (122.2,209.81) -- (279.21,209.81) -- (279.21,172.4) -- (122.2,172.4) -- cycle ;
\draw    (291.2,182.4) -- (350.21,182.8) ;
\draw [shift={(352.21,182.81)}, rotate = 180.39] [color={rgb, 255:red, 0; green, 0; blue, 0 }  ][line width=0.75]    (10.93,-3.29) .. controls (6.95,-1.4) and (3.31,-0.3) .. (0,0) .. controls (3.31,0.3) and (6.95,1.4) .. (10.93,3.29)   ;
\draw    (352.2,202.4) -- (291.21,201.83) ;
\draw [shift={(289.21,201.81)}, rotate = 360.53999999999996] [color={rgb, 255:red, 0; green, 0; blue, 0 }  ][line width=0.75]    (10.93,-3.29) .. controls (6.95,-1.4) and (3.31,-0.3) .. (0,0) .. controls (3.31,0.3) and (6.95,1.4) .. (10.93,3.29)   ;
\draw   (370.2,207.81) -- (527.21,207.81) -- (527.21,170.4) -- (370.2,170.4) -- cycle ;
\draw    (176.2,161.4) -- (176.2,120.81) ;
\draw   (169.71,130.4) -- (176.2,119.81) -- (182.7,130.4)(169.71,126.4) -- (176.2,115.81) -- (182.7,126.4) ;

\draw    (217.21,116.81) -- (217.21,157.4) ;
\draw   (223.7,147.81) -- (217.2,158.4) -- (210.71,147.81)(223.7,151.81) -- (217.2,162.4) -- (210.71,151.81) ;

\draw    (429,161.4) -- (429,120.81) ;
\draw   (422.51,130.4) -- (429,119.81) -- (435.5,130.4)(422.51,126.4) -- (429,115.81) -- (435.5,126.4) ;

\draw    (470.01,116.81) -- (470.01,157.4) ;
\draw   (476.5,147.81) -- (470,158.4) -- (463.51,147.81)(476.5,151.81) -- (470,162.4) -- (463.51,151.81) ;

\draw (141.2,73.4) node [anchor=north west][inner sep=0.75pt]  [font=\scriptsize] [align=left] {Presymplectic-Nijenhuis\\groupois ($\displaystyle \mathcal{G} ,\ \omega ,\ J)$};
\draw (378.6,72.6) node [anchor=north west][inner sep=0.75pt]  [font=\scriptsize] [align=left] {Holomorphic presymplectic\\groupoids ($\displaystyle \mathcal{G} ,\ \Omega =\omega -i\ \omega _{J})$};
\draw (315.2,59.4) node [anchor=north west][inner sep=0.75pt]  [font=\scriptsize] [align=left] {$\displaystyle \Phi $};
\draw (302.2,106.4) node [anchor=north west][inner sep=0.75pt]  [font=\scriptsize] [align=left] {Real part\\};
\draw (315.2,163.4) node [anchor=north west][inner sep=0.75pt]  [font=\scriptsize] [align=left] {$\displaystyle \Phi $};
\draw (302.2,210.4) node [anchor=north west][inner sep=0.75pt]  [font=\scriptsize] [align=left] {Real part\\};
\draw (149.2,178.4) node [anchor=north west][inner sep=0.75pt]  [font=\scriptsize] [align=left] {Dirac-Nijenhuis\\manifolds $\displaystyle ( M,L,r)$};
\draw (403.2,177.4) node [anchor=north west][inner sep=0.75pt]  [font=\scriptsize] [align=left] {Holomorphic Dirac\\manifolds $\displaystyle ( M,\ \mathcal{L}) \ $};
\draw (114,141.8) node [anchor=north west][inner sep=0.75pt]  [font=\tiny] [align=left] {Integration};
\draw (226.6,141.8) node [anchor=north west][inner sep=0.75pt]  [font=\tiny] [align=left] {Differentiation};
\draw (370,141.8) node [anchor=north west][inner sep=0.75pt]  [font=\tiny] [align=left] {Integration};
\draw (479.4,141.8) node [anchor=north west][inner sep=0.75pt]  [font=\tiny] [align=left] {Differentiation};
\end{tikzpicture}
\end{figure}

So one can derive Theorems~\ref{thm:diffpresgrp} and \ref{thm:integholDirac} from Theorem~\ref{thm:main_integration}.
This viewpoint generalizes \cite[Thm.~3.22]{LMX2} in the case of holomorphic Poisson structures and holomorphic symplectic groupoids, through different methods.

An interesting class of holomorphic Dirac structures is given by the  {\em affine Dirac structures on complex Poisson Lie groups} studied in \cite{BIL}, where an explicit description of their holomorphic presymplectic groupoids is presented together with an application to the construction of holomorphic symplectic groupoids integrating holomorphic Poisson homogeneous spaces.

\appendix

\section{Proof of Theorem \ref{thm:bijection}} \label{subsec:proof}
We shall assume some familiarity with double vector bundles (as e.g. in \cite{Mac-book}).

Consider the identification $TM\oplus E \to TE|_{TM}\subseteq TE$ as vector bundles over $TM$.
Any element $U$ of $TE$ can be written (not uniquely) as 
$$
U = Tu(X) +_{\scriptscriptstyle{T}} h(X,e),
$$
for some $u \in \Gamma(E), \, e \in E$, where $+_{\scriptscriptstyle{T}}$ stands for the addition on the fibers of $TE \to TM$. The following lemma gives an explicit formula relating a linear $(1,1)$-tensor with its corresponding 1-derivation.
\begin{lem}
If $K: TE \to TE$ is a linear $(1,1)$-tensor with associated 1-derivation $(D,l,r)$, then 
\begin{align}
\label{eq:K_D_explicit} K(Tu(X)) & = Tu(r(X)) +_{\scriptscriptstyle{T}} h(r(X),D_X(u))\\
\label{eq:K_h_explicit} K(h(X,v)) & = h(r(X), l(v)),
\end{align}
where $u, v \in \Gamma(E)$, $X \in \mathfrak{X}(M)$.
\end{lem}

\begin{proof}
The proof of \eqref{eq:K_h_explicit} follows directly from \eqref{eq:1-1relations}. The proof of \eqref{eq:K_D_explicit} is more involved and we will need to recall some facts from \cite{HT}. Consider the reversal isomorphism $\mathcal{R}: T^*(E^*) \to T^*E$, defined as
$$
\mathcal{R}(\varphi, \beta, e) = (e, -\beta, \varphi),
$$
under the decompositions $T^*(E^*) \cong E^* \oplus T^*M \oplus E$ and $T^*E \cong E \oplus T^*M \oplus E^*$ coming from local trivializations (see \cite[\S~9.5]{Mac-book} for further details). From \cite[Lemma~4.9(a)]{HT}, one has that
\begin{align*}
\<\xi, D_X(u)\> & = \<K(Tu(X)),\mathcal{R}(d\ell_u(\xi))\>_E
\end{align*}
where $u \in \Gamma(E)$, $\xi \in \Gamma(E^*)$, $X \in \frakx(M)$, $\ell_u \in C^{\infty}(E^*)$ is the fiberwise linear function associated with $u$, and $\<\cdot, \cdot\>_E$ is the natural pairing between $TE$ and $T^*E$. By writing $K = Tu(r(X)) +_{\scriptscriptstyle{T}} h(r(X), c)$, for some $c \in \Gamma(E)$, and using basic properties of the cotangent prolongation bundle $T^*E \to E^*$ (see \cite[\S~9.4]{Mac-book}, one has that
\begin{align*}
    \<\xi, D_X(u)\> & = \<Tu(r(X)) +_{\scriptscriptstyle{T}} h(r(X), c),\mathcal{R}(d\ell_u(\xi)) +_{\scriptscriptstyle{E^*}} 0_{\xi}\>_E\\
     & = \<Tu(r(X)), \mathcal{R}(d\ell_u(\xi))\>_E + \<h(r(X), c), 0_\xi\>_E,
\end{align*}
where $+_{\scriptscriptstyle{E^*}}$ is the fiberwise sum and $0_\xi$ is the zero of the cotangent prolongation $T^*E \to E^*$. As $\<h(r(X), c), 0_\xi\>_E = \<\xi, c\>$, it only remains to prove that $\<Tu(r(X)), \mathcal{R}(d\ell_u(\xi))\>_E=0$. This will follow from  (see \cite[Prop.~9.5.3]{Mac-book})
$$
\mathcal{R}(d\ell_u(\xi(x))) = d\ell_\xi(u(x)) - q^*(d\<\xi, u\>)(u(x)), \,\,\, x \in M.
$$
Indeed, this identity implies that
\begin{align*}
\<Tu(r(X)), \mathcal{R}(d\ell_u(\xi))\>_E & =\Lie_X\<\xi, u\> - \Lie_X\<\xi, u\> = 0.  
\end{align*}
\end{proof}

To complete the proof of Theorem \ref{thm:bijection}, notice that $\Phi$ being a vector bundle map implies that
$$
T\Phi(h(X,e_1)) = h(X,\Phi(e_1)),
$$
for $e_1 \in E_1$. It follows from \eqref{eq:K_h_explicit} that 
$$
T\Phi \circ K_1(h(X,e_1)) = K_2(T\Phi(h(X,e_1))) \Longleftrightarrow r_1 = r_2 =r,\,\,\,\ \Phi \circ l_1 = l_2 \circ \Phi.
$$
Similarly, from \eqref{eq:K_D_explicit},
\begin{align*}
T\Phi(K_1(Tu(X))) & = T(\Phi(u))(r(X)) +_{\scriptscriptstyle{T}} h(r(X), \Phi(D^1_X(u))),\\
K_2(T\Phi(Tu(X))) & = T(\Phi(u))(r(X))  +_{\scriptscriptstyle{T}} h(r(X), D^2_X(\Phi(u))).
\end{align*}
Therefore
$$
T\Phi(K_1(Tu(X))) = K_2(T\Phi(Tu(X))) \Longleftrightarrow \Phi\circ D^1_X = D^2_X\circ \Phi.
$$


\section{Other notions of Dirac-Nijenhuis structures}
\label{app:B}
In this section, we relate our notion of Dirac-Nijenhuis structure with other existing definitions in the literature \cite{CG,  Nunes04, LB}. Our main tool to compare the different definitions is the bracket on $\T M$ defined by
\begin{align}
\label{dfn:D_bracket}\Cour{(X,\alpha), (Y,\beta)}_{\D^r} & = \Cour{(r(X),r^*(\alpha)), (Y,\beta)} + \D^r_{Y}(X,\alpha)\\
 \nonumber & = ([X,Y]_r, \mathcal{L}_{r(X)}\beta - i_{r(Y)}d\alpha),
\end{align}
for $r\in \Omega^1(M, TM)$ (and $[\cdot,\cdot]_r$ defined in \eqref{eq:r_bracket}). The properties of $(\T M, \Cour{\cdot,\cdot}_{\D^r})$ will be discussed in \cite{BDC}. Here, we only need two facts about the bracket \eqref{dfn:D_bracket}: (1) if $L$ is a Dirac structure that is compatible with $r$, then $L$ is involutive with respect to $\Cour{\cdot, \cdot}_{\D^r}$; 2) the following identity holds:
\begin{equation}\label{eq:D_bracket_skew}
   \Cour{(X,\alpha), (Y, \beta)}_{\mathbb{D}^r} + \Cour{(Y, \beta),(X,\alpha)}_{\mathbb{D}^r} =  d \< (r(X),r^*(\alpha)), (Y, \beta)\>.
\end{equation}

\subsection*{Dirac-Nijenhuis structures of contraction type} A vector-bundle morphism $N: \T M \rightarrow \T M$ (over the identity on $M$) induces a bracket on sections of $\T M$ called \textit{the contracted bracket}:
\begin{equation}\label{ctr}
   \Cour{\sigma_1, \sigma_2}_N:= \Cour{N \sigma_1, \sigma_2} + \Cour{\sigma_1, N\sigma_2}- N(\Cour{\sigma_1,\sigma_2}) 
\end{equation}
for all $\sigma_1,\sigma_2 \in \Gamma(\T M)$. One can also define a Nijenhuis-type torsion by
\begin{equation}
T_N(\sigma_1,\sigma_2) = N(\Cour{\sigma_1,\sigma_2}_N) - \Cour{N\sigma_1,N\sigma_2}.    
\end{equation}

To distinguish from our definition of Dirac-Nijenhuis structure, we shall say that a Dirac structure $L \subset \T M$ and $N$ as above form a \textit{Dirac-Nijenhuis structure of contraction type} if $\Cour{\cdot,\cdot}_N$ restricts to skew-symmetric bracket on $\Gamma(L)$ and $T_N|_{\Gamma(L)}  = 0$. Dirac-Nijenhuis structures of contraction type were defined and studied in \cite{CG,Nunes04}.

\begin{prop} \label{propGrab} Let $L$ be a Dirac  structure on $M$, $r\in \Omega^1(M,TM)$, and consider $N=(r,0)$. Then $(L,N)$  is a Dirac-Nijenhuis structure of contraction type if and only if 
\begin{itemize}
    \item[(i)] $(r,r^*)(L) \subseteq L$,
    \item[(ii)]$\D^r_Y(\Gamma(L)) \subseteq \Gamma(L)$, for every $Y \in \pr_T(L)$,
    \item[(iii)] $\N_r(X,Y)=0$, for every $X,Y \in \pr_T(L)$.
\end{itemize}
In particular, if $(L,r)$ is Dirac-Nijenhuis, then $(L, N)$ is Dirac-Nijenhuis of contraction type.
\end{prop}
\begin{proof}
It is a direct verification that, for $N=(r,0)$,
$
    \Cour{\cdot, \cdot}_N = \Cour{\cdot, \cdot}_{\mathbb{D}^r}.
$
Hence, it follows from \eqref{eq:D_bracket_skew} that $\Cour{\cdot,\cdot}_N$ is skew-symmetric on $L$ if and only if 
$$
d\<(r,r^*)(\sigma_1), \sigma_2\> = 0, \,\,\; \forall \, \sigma_1,  \sigma_2 \in \Gamma(L).
$$
We claim that this is equivalent to $(r,r^*)(L) \subset L$. Indeed, write $\T M = L \oplus L'$, where $L'$ is a lagrangian complement of $L$, and take a local basis $\{\sigma_1, \dots, \sigma_m, \sigma^1, \dots, \sigma^m\}$, where 
$$
\sigma_i \in \Gamma(L), \,\; \sigma^j \in \Gamma(L'), \; \<\sigma_i, \sigma^j\> = \delta_i^j, \,\,\;\;\; i, j =1, \dots, m.
$$
By writing $(r,r^*)(\sigma_k) = a_k^i \sigma_i + b_{kj}\sigma^j$, one sees that $b_{kj}$ must be zero since
$
d\<\sigma_k, f\sigma_j\> =0, \,\, \forall \, f \in C^\infty_{\mathrm{loc}}(M).  
$
Now, once $(r,r^*)(L)\subseteq L$, it follows from \eqref{dfn:D_bracket} that 
$$
\<\Cour{(X,\alpha), (Y,\beta)}_{\D^r}, (Z, \gamma)\> = \<\D^r_Y(X,\alpha), (Z,\gamma)\>, \,\,\;\; \forall\, (X,\alpha), \, (Y, \beta), \, (Z, \gamma) \in \Gamma(L).
$$
Since $\Cour{\cdot,\cdot}_N=\Cour{\cdot, \cdot}_{\D^r}$ and $L$ is lagrangian, it follows that $L$ is involutive  if and only if $\D^r_Y(\Gamma(L)) \subseteq \Gamma(L)$, for every $Y \in \pr_T(L)$. At last, condition (iii) is a consequence of 
$$
T_N((X,\alpha), (Y,\beta)) = (\N_r(X,Y), 0).
$$
\end{proof}

In the particular case of Poisson structures, Dirac structures of contraction type lead to objects which are strictly more general than  Poisson-Nijenhuis structures, in contrast with what is described in Example~\ref{exam:pn_dn}.
Indeed, when $L= \mathrm{graph}(\pi)$ for a Poisson structure $\pi$,  Proposition \ref{propGrab} says that $(L,r)$ is Dirac-Nijenhuis of contraction type if and only if $\pi^\sharp \circ r^* = r \circ \pi^\sharp$,
$$
\pi^\sharp(C_\pi^r(\alpha, \beta)) = 0,
$$
and $\N_r(X,Y)=0$ for $X, Y \in \pi^\sharp(TM)$, where $C_\pi^r$ is the concomitant \eqref{dfn:schw_concomitant}.

\subsection*{Dirac-Nijenhuis structures of double type} A Nijenhuis operator $r: TM \to TM$ makes the pair $(TM, T^*M)$ into a Lie bialgebroid, where $T^*M$ has the trivial Lie algebroid structure and $TM$ has $[\cdot,\cdot]_r$ as Lie bracket and anchor $r$. The corresponding double is the Courant algebroid structure $\T M_r = TM \oplus T^*M$ with the natural pairing, anchor given by $r \circ \pr_T$, and bracket given by
\begin{equation}
    \Cour{(X,\alpha), (Y, \beta)}_{dbl} :=([X,Y]_r, \mathcal{L}_X^r\beta-i_Y d^r \alpha),
\end{equation}
where $\mathcal{L}^r$ and $d^r$ denote the Lie derivative and differential associated with the Lie algebroid $(TM,[\cdot,\cdot]_r, r)$,  respectively.

We say that a pair $(L, r)$  is a \textit{Dirac-Nijenhuis structure of double type} if $r$ is a Nijenhuis operator, and $L$ and $L_{(1,0)}=(r, \id_{T^*M})(L)$ are Dirac structures in $\mathbb{T}M$ and $\mathbb{T}M_r$, simultaneously. This notion of Dirac-Nijenhuis structure was studied in \cite{LB} in greater generality.

\begin{prop}\label{propLiu} If $(L,r)$ is a Dirac-Nijenhuis structure and $(r,\id_{T^*M})|_L$ is injective, then $(L, r)$ is a Dirac-Nijenhuis structure of double type.
\end{prop}
\begin{proof}
One can check that
$$
i_Yd^r\alpha = i_{r(Y)}d\alpha + \<D^{r,*}_{(\cdot)}(\alpha), Y\>, \;\;\;\,\,\, \Lie_X^r \beta = \Lie_{r(X)}\alpha - \<\alpha, D^r_{(\cdot)}(X)\>.
$$
It follows that
$$
    \Cour{(X,\alpha), (Y,\beta)}_{dbl}= \Cour{(X,\alpha),(Y,\beta)}_{\mathbb{D}^r} - \< \mathbb{D}^r_{(\cdot)}(X,\alpha), (Y, \beta)\>,
$$
where  $\Cour{\cdot, \cdot}_{\mathbb{D}^r}$ is defined by \eqref{dfn:D_bracket}.

Therefore, if $(L,r)$ is a Dirac-Nijenhuis structure, since $\<\D_{(\cdot)}^r(\sigma_1), \sigma_2\>=0$, for all $\sigma_1, \sigma_2 \in \Gamma(L)$, one has that
$
\Cour{\cdot, \cdot}_{dbl}|_{\Gamma(L)}=  \Cour{\cdot,\cdot}_{\mathbb{D}}|_{\Gamma(L)}.
$
This implies that $L$ is involutive under $\Cour{\cdot,\cdot}_{dbl}$. 
Also, if $(r,\id_{T^*M})|_L$ is injective, it follows from Prop. \ref{prop:DNhier} that $(L_{(1,0)},r)$ is a Dirac-Nijenhuis structure on $\mathbb{T}M$ and we can apply the same argument above to obtain that $L_{(1,0)}$ is a Dirac structure on $\T M_r$. This concludes the proof.
\end{proof}

We remark that, just as for Dirac structures of contraction type, Dirac structures of double type are also strictly more general than Poisson-Nijenhuis structures when  $L=\mathrm{graph}(\pi)$, for a Poisson structure $\pi$. Indeed, in this case it is proved in \cite{LB} that $(L, r)$ is Dirac-Nijenhuis of double type if and only if $r$ is Nijenhuis, $\pi_r^\sharp = r \circ \pi^\sharp$, and 
$$
[\pi, \pi_r]=0.
$$
In general, the condition  $[\pi,\pi_r]=0$ is weaker than the vanishing of the concomitant $C_\pi^r$ (they are equivalent if $\pi$ is symplectic, see \cite[Thm.~2.1]{Vai96}).



\begin{thebibliography}{10}


\bibitem{bonechi15}
F.~Bonechi. Multiplicative integrable models from Poisson-Nijenhuis structures. {\em From Poisson brackets to universal quantum symmetries}, 19-–33, Banach Center Publ., 106, Polish Acad. Sci. Inst. Math., Warsaw, 2015.

\bibitem{BCQT14}
F.~ Bonechi, N.~ Ciccoli, J.~ Qiu, and M.~ Tarlini. Quantization of Poisson manifolds from the integrability of the modular function. {\em Comm. Math. Phys.} {\bf 331} (2014), 851-–885.

\bibitem{BQT18}
F.~ Bonechi, J. Qiu, and M. Tarlini. Complete integrability from Poisson-Nijenhuis structures on compact hermitian symmetric spaces. {\em J. Symplectic Geom.} {\bf 16} (2018), 1167-–1208. 




\bibitem{HB}
H.~Bursztyn.
\newblock A brief introduction to {D}irac manifolds.
\newblock {\em Geometric and Topological Methods for Quantum Field Theory:
  Proceedings of the 2009 Villa de Leyva Summer School. Cambridge University
  Press} (2013), 4--38,

\bibitem{HA}
H.~Bursztyn and A.~Cabrera.
\newblock Multiplicative forms at the infinitesimal level.
\newblock {\em Math. Ann.} {\bf 353} (2012), 663--705

\bibitem{bcwz}
H.~Bursztyn, M.~Crainic, A.~Weinstein, and C.~Zhu.
\newblock Integration of twisted Dirac brackets.
\newblock {\em Duke Math. J.} {\bf 123} (2003), 549--607.

\bibitem{HT}
H.~Bursztyn and T.~Drummond.
\newblock {L}ie theory of multiplicative tensors.
\newblock {\em Math. Ann.} {\bf 375} (2019), 1489--1554.

\bibitem{BDC}
H.~Bursztyn, T.~Drummond, and C.~Netto.
\newblock Courant-{N}ijenhuis algebroids.
\newblock {\em In preparation}, 2021.

\bibitem{BIL}
H.~Bursztyn, D.~Iglesias-Ponte, and J-H. Lu.
\newblock Dirac geometry and integration of poisson homogeneous spaces.
\newblock {\em Preprint arXiv:1905.11453}, (2019).

\bibitem{BurRadko}
H.~Bursztyn and O.~Radko.
\newblock Gauge equivalence of Dirac structures and symplectic groupoids.
\newblock {\em Ann. Inst. Fourier (Grenoble)}, {\bf 53} 2003, 309--337.

\bibitem{Crampin}
F.~Cantrijn, M.~Crampin, and W.~Sarlet.
\newblock Lifting geometric objects to a cotangent bundle, and the geometry of
  the cotangent bundle of a tangent bundle.
\newblock {\em J. Geom. Phys.} {\bf 4} (1987), 469--492.

\bibitem{CG}
J. F. Cari\~{n}ena, J. Grabowski, and G. Marmo.
\newblock Courant algebroid and {L}ie bialgebroid contractions.
\newblock {\em J. Phys. A} {\bf 37} (2004), 5189--5202.


\bibitem{catfel}
A.~Cattaneo and G.~Felder.
\newblock Poisson sigma models and symplectic groupoids.
\newblock {\em Quantization of singular symplectic quotients}, 61--93, {\em Progr. Math.} {\bf 198}, {\ Birkh\"auser,
Basel}, 2001.

\bibitem{Nunes04}
J.~Clemente-Gallardo and J.~M. Nunes~da Costa.
\newblock Dirac-{N}ijenhuis structures.
\newblock {\em J. Phys. A} {\bf 37} (2004), 7267--7296.

\bibitem{CDW}
A. Coste, P. Dazord and A. Weinstein: \textit{Groupo\"ides
symplectiques}. Publications du D\'epartement de Math\'ematiques.
Nouvelle S\'erie. A, Vol. 2,  i--ii, 1--62, Publ. D\'ep. Math.
Nouvelle S\'er. A, 87-2, Univ. Claude-Bernard, Lyon, 1987.

\bibitem{Cour}
T.~Courant.
\newblock Dirac manifolds.
\newblock {\em Trans. Amer. Math. Soc.} {\bf 319} 1990, 631--661.

\bibitem{crainic}
M.~ Crainic. \newblock Generalized complex structures and Lie brackets. 
{\em Bull. Braz. Math. Soc. (N.S.)} {\bf 42} (2011), 559-–578.

\bibitem{CF}
M.~Crainic and R.~Fernandes.
\newblock Integrability of Poisson brackets.
\newblock {\em J. Differential Geom.}{ \bf 66} (2004), 71--137.

\bibitem{damfer08}
P. A. ~ Damianou and R. L.~ Fernandes. Integrable hierarchies and the modular class. {\em Ann. Inst. Fourier (Grenoble)} {\bf 58} (2008), 107-–137.

\bibitem{DasOku}
A.~Das and S.~Okubo.
\newblock A systematic study of the {T}oda lattice.
\newblock {\em Ann. Physics} {\bf 190},  1989, 215--232.


\bibitem{T}
T.~Drummond.
\newblock Lie-{N}ijenhuis bialgebroids.
\newblock {\em Preprint arXiv:2004.10900}, (2020).

\bibitem{FMOP}
G.~Falqui, I.~Mencattini, G.~Ortenzi, and M.~Pedroni.
\newblock Poisson quasi-{N}ijenhuis manifolds and the {T}oda system.
\newblock {\em Math. Phys. Anal. Geom.}, {\bf 23} (2020), Paper No. 26, 17 pp.



\bibitem{Frejlich}
P.~Frejlich.
\newblock Concurring {D}irac structures.
\newblock {\em Talk at International conference on {P}oisson geometry, IMPA},
  2019.
  
\bibitem{GelfandDorfman}
I. M.~Gelfand and I. Ya. Dorfman, \newblock Hamiltonian operators and algebraic structures related to them,
{\em Funct. Anal. Appl.}, {\bf 13} (1979), 248-–262.


\bibitem{Gual14}
M.~Gualtieri.
\newblock Generalized {K}\"{a}hler geometry.
\newblock {\em Comm. Math. Phys.}, {\bf 331} (2014), 297--331.

\bibitem{LB}
L.-G. He and B.-K. Liu.
\newblock Dirac-{N}ijenhuis manifolds.
\newblock {\em Rep. Math. Phys.}{\bf 53} (2004), 123--142.

\bibitem{LB2}
L.-G. He and B.-K. Liu.
\newblock Some properties of {D}irac-{N}ijenhuis manifolds.
\newblock {\em Rep. Math. Phys.} {\bf 58} (2006), 165--194.

\bibitem{KMS}
I.~Kol\'{a}\v{r}, P.~Michor, and J.~Slov\'{a}k.
\newblock {\em Natural operations in differential geometry}.
\newblock Springer-Verlag, Berlin, 1993.

\bibitem{rubtsov}
 S.~ Khoroshkin, A. ~Radul, and V.~ Rubtsov, A family of Poisson structures on
hermitian symmetric spaces. {\em Comm. Math. Phys.} {\bf 152} (1993), 299--315.



\bibitem{Kos96}
Y.~Kosmann-Schwarzbach.
\newblock The {L}ie bialgebroid of a {P}oisson-{N}ijenhuis manifold.
\newblock {\em Lett. Math. Phys.}{\bf  38} (1996), 421--428.



\bibitem{KS-Magri90}
Y.~Kosmann-Schwarzbach and F.~Magri.
\newblock Poisson-{N}ijenhuis structures.
\newblock {\em Ann. Inst. H. Poincar\'{e} Phys. Th\'{e}or.} {\bf 53} (1990), 35--81.

\bibitem{LMX2}
C.~Laurent-Gengoux, M.~Sti\'enon, and P.~Xu.
\newblock Holomorphic {P}oisson manifolds and {L}ie algebroids.
\newblock {\em Int. Math. Res. Not.} {\bf 2008}, Art.ID rnn 088, 46 pp.

\bibitem{SXC}
C.~Laurent-Gengoux, M.~Stienon, and P.~Xu.
\newblock Integration of holomorphic {L}ie algebroids.
\newblock {\em Math. Ann.} {\bf 345} (2009), 895--923.

\bibitem{Mac-book}
Mackenzie, K., {\em General theory of Lie groupoids and Lie
algebroids}. London Mathematical Society Lecture Note Series, 213.
Cambridge University Press, Cambridge, 2005.


\bibitem{Mag}
F.~Magri.
\newblock A simple model of the integrable Hamiltonian equation.
\newblock {\em J. Math. Phys.} {\bf  19} (1978), 1156--1162.

\bibitem{MagMars}
F.~Magri and T.~Marsico.
\newblock Some developments of the concepts of {P}oisson manifolds in the sense
  of {A}. {L}ichnerowicz.
\newblock In: {\em Gravitation, Electromagnetism, and Geometric Structures} (G. Ferrarese, ed.), Pitagora Editrice, Bologna (1996), pp. 207--222.

\bibitem{MM}
F.~Magri and C.~Morosi.
\newblock A geometrical characterization of integrable {H}amiltonian systems
  through the theory of {P}oisson-{N}ijenhuis manifolds.
\newblock {\em Quaderno, University of Milan}, 19, 1984.



\bibitem{Mein}
E.~Meinrenken, Poisson geometry from a Dirac perspective. {\em Lett. Math. Phys.} {\bf 108} (2018), 447--498.

\bibitem{Ra}
J.~Rawnsley.
\newblock Flat partial connections and holomorphic structures in ${C}^\infty$
  vector bundles.
\newblock {\em Proc. Amer. Math. Soc.} {\bf  73} (1979), 391--397.

\bibitem{SW}
P.~\v{S}evera and A.~Weinstein.
\newblock Poisson geometry with a 3-form background. Noncommutative geometry and string theory (Yokohama, 2001)
\newblock {\em   Progr. Theoret. Phys. Suppl.} {\bf 144} (2001), 145--154.



\bibitem{SX}
M.~Stienon and P.~Xu.
\newblock Poisson quasi-{N}ijenhuis manifolds.
\newblock {\em Comm. Math. Phys.} {\bf 270} (2006), 709-725.

\bibitem{Vai96}
I. Vaisman.
\newblock Complementary {$2$}-forms of {P}oisson structures.
\newblock {\em Compositio Math.}, {\bf 101} (1996), 55--75q.

\bibitem{YI}
K.~Yano and S.~Ishihara.
\newblock {\em Tangent and cotangent bundles}.
\newblock Marcel Dekker, Inc., New York, 1973.

\end{thebibliography}
\end{document}